%% file: main.tex
\pdfoutput=1

\PassOptionsToPackage{linktocpage=true}{hyperref}

\documentclass[notheorems]{colt2023}

\input{prologueCOLT.tex}

\usepackage{mathrsfs}		
\usepackage{dsfont}

\usepackage{url}					
\hypersetup{colorlinks=true, urlcolor=blue, linktoc = all, pdfborder={0 0 0}}

\usepackage{hhline}
\usepackage{silence}
\usepackage{xparse}
\usepackage{xargs}
\usepackage{etoolbox}
\usepackage{svg}
\usepackage{soul}
\usepackage{tablefootnote}
\usepackage[bb=boondox]{mathalpha} 
\definecolor{lightgrey}{rgb}{0.9,0.9,0.9}
\sethlcolor{lightgrey}
\definecolor{mygray}{rgb}{0.6,0.6,0.6}

\input{make_cleveref_work.tex}

\newtheorem{theorem}{Theorem}
\newtheorem{lemma}[theorem]{Lemma}
\newtheorem{proposition}[theorem]{Proposition}
\newtheorem{remark}[theorem]{Remark}

\newtheorem{fact}[theorem]{Fact}

\crefname{enumi}{Property}{Properties} 
\usepackage{times}

\let\epsilon\varepsilon

\usepackage{pifont}

\makeatletter
\renewcommand\paragraph{\@startsection{paragraph}{4}{\z@}%
                                    {0ex \@plus0.5ex \@minus.2ex}%
                                    {-1em}%
                                    {\normalfont\normalsize\bfseries}}
\makeatother

\newcommand\blfootnote[1]{
  \begingroup
  \renewcommand\thefootnote{}\footnote{#1}%
  \addtocounter{footnote}{-1}%
  \endgroup
}

\title[Accelerated and Sparse Algorithms for Approximate Personalized PageRank and Beyond]{Accelerated and Sparse Algorithms for\\ Approximate Personalized PageRank and Beyond}

\coltauthor{%
 \Name{David Martínez-Rubio} \Email{\href{mailto:martinez-rubio@zib.de}{martinez-rubio@zib.de}}\\
 \addr Zuse Institute Berlin, and \\ Institute of Mathematics, Berlin Institute of Technology\\ Berlin, Germany%
 \AND
 \Name{Elias Wirth} \Email{\href{mailto:wirth@math.tu-berlin.de}{wirth@math.tu-berlin.de}}\\
 \addr Zuse Institute Berlin, and \\ Institute of Mathematics, Berlin Institute of Technology\\ Berlin, Germany%
 \AND
 \Name{Sebastian Pokutta} \Email{\href{mailto:pokutta@zib.de}{pokutta@zib.de}}\\
 \addr Zuse Institute Berlin, and \\ Institute of Mathematics, Berlin Institute of Technology\\ Berlin, Germany%
}

\input{algorithm_config.tex}

\input{bibliography_config.tex}
\input{definitions.tex}

\begin{document}

\maketitle

\blfootnote{Most of the notations in this work have a link to their definitions. For example, if you click or tap on any instance of $\xast$, you will jump to the place where it is defined as the minimizer of the function we consider in this work.}

\begin{abstract}
    It has recently been shown that \ISTA{}, an unaccelerated optimization method, presents sparse updates for the $\ell_1$-regularized personalized PageRank problem, leading to cheap iteration complexity 
    and providing the same guarantees as the approximate personalized PageRank algorithm (\appr{}) \citep{fountoulakis2019variational}. 
    In this work, we design an accelerated optimization algorithm for this problem that also performs sparse updates, providing an affirmative answer to the COLT 2022 open question of \citet{fountoulakis2022open}.
    Acceleration provides a reduced dependence on the condition number, 
    while the dependence on the sparsity in our updates differs from the \ISTA{} approach. 
    Further, we design another algorithm by using conjugate directions to achieve an exact solution while exploiting sparsity. Both algorithms lead to faster convergence for certain parameter regimes. Our findings apply beyond PageRank and work for any quadratic objective whose Hessian is a positive-definite $M$-matrix.
\end{abstract}

\section{Introduction}\label{sec:introduction}

\emph{Graph clustering}, the process of dividing a graph into subclusters that are internally similar or connected in some application-specific sense \citep{schaeffer2007graph}, has been widely applied in various domains, including technical \citep{virtanen2003clustering, andersen2006local}, biological \citep{xu2002clustering, bader2003automated, boyer2005syntons}, and sociological \citep{newman2003properties, traud2012social} settings. With the advent of large-scale networks, traditional approaches that require access to the entire graph have become infeasible \citep{jeub2015think, leskovec2009community, fortunato2016community}. This trend has led to the development of \emph{local graph clustering algorithms}, which only visit a small subset of vertices of the graph \citep{andersen2006local, andersen2008algorithm, mahoney2012local, spielman2013local, kloster2014heat, orecchia2014flow, veldt2016simple, wang2017capacity, yin2017local, fountoulakis2019variational}.

 At the heart of the study of these algorithms lies the \emph{approximate personalized PageRank algorithm} (\newtarget{def:acronym_approximate_personalized_page_rank}{\appr{}}) \citep{andersen2006local}, which approximates the solution of the PageRank linear system \citep{page1999page} and rounds the approximate solution to find local partitions in a graph. The \appr{} algorithm was introduced only from an algorithmic perspective, that is, its output is determined only algorithmically and not formulated as the solution to an optimization problem. Thus, quantifying the impact of heuristic modifications on the method is difficult, see, for example, \citep{gleich2014anti}.
Recently, \citet{fountoulakis2019variational} proposed a variational formulation of the local graph clustering problem as an $\ell_1$-regularized convex optimization problem, which they solved using the \emph{iterative shrinkage-thresholding algorithm} (\newtarget{def:acronym_ista}{\ista{}}) \citep{parikh2014proximal}. In this problem, \ista{} was shown to exhibit local behaviour, which leads to a running time that only depends on the nodes that are part of the solution and its neighbors, and is independent of the size of the graph. \citet{fountoulakis2022open} raised the open question of whether accelerated versions of the \ista{}-based approach or other acceleration techniques, for example, the \emph{fast iterative shrinkage-thresholding algorithm} (\newtarget{def:acronym_fista}{\fista{}}) \citep{parikh2014proximal}, or \emph{linear coupling} \citep{allenzhu2019nearly}, could lead to faster local graph clustering algorithms. 
In particular, \ista{} enjoys low per-iteration complexity since its iterates are at least as sparse as the solution, and the question is whether we can attain acceleration and reduce the dependence on the condition number on the computational complexity, while keeping sparse per-iteration updates. 

\paragraph{Sparse Algorithms and Acceleration.}
In this work, we answer the question in the affirmative. We first study the problem beyond acceleration and propose a method based on conjugate directions that optimizes exactly and is faster than \ista{} and our accelerated algorithm in some parameter regimes. Then, we show that we can implement an approximate version of the previous method by means of acceleration while performing sparse updates, which leads to faster convergence for ill-conditioned problems, among others. See \cref{table:comparisons:riemannian} for a summary of the complexities of our algorithms and of prior work, and see \cref{sec:algorithmic_comparisons} for a discussion comparing these complexities. 
Our algorithms sequentially determine the coordinates in the support of the solution. The main differences between the two approaches are that the conjugate-directions-based approach solves the problem in increasing subspaces exactly and requires to incorporate new coordinates one by one, while the accelerated algorithm solves this approximately and can add any number of new coordinates at a time.
Beyond the PageRank problem, our algorithms apply generally to the quadratic problem $\min_{x\in\Rp^{\n}}\{\g(\xx)\defi\innp{\xx, \Q\xx} - \innp{\bb, \xx}\}$, where $\Q$ is a symmetric positive-definite $M$-matrix.

\paragraph{Problem Structure.} The rates achieved with our two methods exploit improved geometric understanding of the $\ell_1$-regularized PageRank problem structure that we present. In particular, the $\ell_1$-regularized problem can be posed as a problem constrained to the positive orthant $\R^{\n}_{\geq 0}$. Based on this formulation, we characterize a region of points for which a negative gradient coordinate $i$ indicates $i$ is in the support $\suppast$ of the optimal solution $\xast$, provide sufficient conditions for finding points in this region with negative gradient coordinates, and show coordinatewise monotonicity of minimizers restricted to some relevant increasing subspaces, among other things.

\begin{table}[ht!]
    \centering
    \caption{Convergence rates of different algorithms exploiting sparsity for the $\ell_1$-regularized PageRank problem and other more general quadratic optimization problems with Hessian $\Q$, condition number $\L/\alpha$, $\suppast \defi \supp(\xast)$,  $\vol(\suppast) \defi \nnz(\Q_{:,\suppast})$ and $\intvol(\suppast) \defi \nnz(\Q_{\suppast,\suppast})$. } 
    \label{table:comparisons:riemannian} 
\begin{tabular}{llc} 
    \toprule
    \textbf{Method}   &  \textbf{Time complexity} & \textbf{Space complexity} \\
    \midrule
    \midrule
    \ista{} \citep{fountoulakis2019variational}                 & $\bigotilde{\volast\frac{\L}{\alpha}}$  & $\bigo{\sparsity}$ \\
    \midrule
     \cdappr{} (\cref{alg:sparse_conjugate_directions}) & $\bigo{\sparsity^3 + \sparsity \vol(\suppast)}$ & $\bigo{\sparsity^2}$\\
    \midrule
    \aspr{} (\cref{alg:sparse_acceleration})             & $\bigotilde{\sparsity\intvol(\suppast)\sqrt{\frac{\L}{\alpha}} + \sparsity \volast}$  & $\bigo{\sparsity}$\\
    \bottomrule
\end{tabular}
\end{table}

\subsection{Other Related Works}\label{sec:related_works}

Our solutions make use of first-order methods: accelerated projected gradient descent \citep{nesterov1998introductory} and the method of conjugate directions \citep{nocedal1999numerical}.
First-order optimization methods are attractive in the high-dimensional regime, due to their fast per-iteration complexity in comparison to higher order methods. In the strongly convex and smooth case, accelerated gradient descent is an optimal first-order method \citep{nesterov1998introductory} and it improves over gradient descent by reducing the dependence on the condition number. Because of this reason, accelerated gradient descent is especially useful for ill-conditioned problems. A method related to the conjugate directions method is the conjugate gradients algorithm  \citep{nocedal1999numerical}. Both of these conjugate methods can work in affine subspaces \citep{gower2014conjugate}, but to the best of our knowledge, it is not know how to provably use these algorithms with other kinds of constraints, see \citep{vollebregt2014bound} and references therein. For quadratic objectives, the conjugate gradient algorithm is also an accelerated method, and it belongs to the family of Krylov subspace methods, of which the generalized minimal residual method is an important example \citep{saad1986gmres}. In fact, the conjugate gradient algorithm was the inspiration for the first nearly-accelerated method for smooth convex optimization by \citet{nemirovski_bubeck}. Conjugate methods have been used to solve linear systems \citep{saad2003iterative} and although these methods are known to exploit the sparsity of the matrix, to the best of our knowledge there are no analyses of conjugate methods that exploit the sparsity of the solution.

For the $\ell_1$-regularized PageRank problem, \citet{hu2020local} demonstrated through numerical experiments  that the updates generated by \fista{} do not exhibit the same level of sparsity as those produced by \ista{} for this type of problem. 
To the best of our knowledge, no other works have studied the open question raised by \citet{fountoulakis2022open}.

\subsection{Preliminaries}\label{sec:preliminaries}

In this section, we introduce some definitions and notation to be used in the rest of this work.

Throughout, let $\n\in\N$. We use $[\n] = \{1, 2, \dots, \n\}$. We use the big-$\mathcal{O}$ notation $\newtarget{def:big_o_tilde}{\bigotilde{\cdot}}$ to omit logarithmic factors.
Let $\oneterm \in\R^{\n}$ denote the all-ones vector.
Denote the support of a vector $\xx\in\R^{\n}$ by $\supp( \xx )= \left\{i\in[\n] \mid x_i \neq 0 \right\}$ and define the projection of $\xx\in \R^{\n}$ onto a convex subset $C\subseteq \R^{\n}$ by $\newtarget{def:projection_operator}{\proj{C}}(\xx)=\argmin_{\yy\in C} \norm{\xx - \yy}_2$.
For $i\in[\n]$, we use $\newtarget{def:vector_of_canonical_basis}{\canonical[i]}\in\R^{\n}$ to denote the $i$-th unit vector and $\newtarget{def:simplex}{\simplex{\n}}$ to denote the $\n$-dimensional simplex. For $S\subseteq [\n]$, and a function    $f\colon \R^{\n} \to \R$, let $\nabla_S f(\xx)$ be the vector containing $(\nabla_i f(\xx))_{i\in S}$ sorted by index.  
Throughout, $\newtarget{def:symm_pos_def_M_matrix_Q}{\Q} \in \mathcal{M}_{\newtarget{def:dimension}{\n}\times \n}(\R)$ is always a positive-definite matrix with non-positive off-diagonal entries, that is, an $M$-matrix such that $\Q \succ 0$.

In this work, for one matrix $\Q$ of the form above and a vector $\bb\in\R^{\n}$, we study the optimization of a quadratic of the form $\newtarget{def:function_g_constrained_version_of_l1_reg_PageRank}{\g}(\xx) \defi \innp{\xx, \Q\xx} - \innp{\bb, \xx}$ constrained to the positive orthant $\Rp^{\n}$. Without loss of generality, we can thus assume that $\Q$ is symmetric. By strong convexity, the solution is unique. In the sequel, we focus on optimization algorithms for this problem whose iterates always have support contained in the support of the optimal solution $\newtarget{def:optimizer}{\xast}=\argmin_{\xx\in\R^n_{\geq 0}} \g(\xx)$. We define $\newtarget{def:support_of_the_solution}{\suppast} \defi \supp(\xast)$. We refer to coordinates $i \in [\n]$ as good if $i \in \suppast$, and as bad otherwise. We denote by $\newtarget{def:smoothness_constant}{\L}$ and $\newtarget{def:strong_convexity_of_g}{\alpha}$ upper and lower bounds on the eigenvalues of $\Q$, that is, smoothness and strong convexity constants of $\g$ defined as above, respectively. In short, we have $0 \prec \alpha \I  \preccurlyeq \nabla^2 \g(\xx) \preccurlyeq  \L \I$, for $\xx \in \R^{\n}$.

Throughout, $\newtarget{def:graph_G}{\G} = (\V, \edges)$ is a graph with vertex and edge sets $\newtarget{def:vertices_of_graph}{\V}$ and $\newtarget{def:edges_of_graph}{\edges}$, respectively.
We assume that $\card{\V} = \n$, that is, $\G$ consists of $\n$ vertices. Given two vertices $i, j\in [\n]$, $i\newtarget{def:neighbor_in_graph}{\neigh} j$ denotes that they are neighbours. For $S\subseteq \V$, $i\neigh S$ indicates that $i$ is the neighbour of at least one node in $S$.
As we describe in the next section, in PageRank problems, the matrix $\Q$ corresponds to a combination of the Lagrangian of a graph and the identity matrix $\newtarget{def:identity_matrix}{\I}$. For a subset of vertices $S\subseteq \V$, we formally define the volume of $S$ as $\newtarget{def:volume}{\vol}(S) = \sum_{i\in S} d_i + \card{S}$, that is, as the sum of the degrees of vertices in $S$, plus $\card{S}$, to account for the regularization, that presents a similar effect to lazyfying the walk given by the graph. Similarly, we formally define the internal volume of $S$ as $\newtarget{def:internal_volume}{\intvol}(S) \defi \card{S} + \sum_{(i,j)\in \edges} \ones_{\{i,j\in S\}} $, that is, as the sum of edges of the subgraph induced by $S$, plus $\card{S}$, to account for the regularization. 
This definition corresponds to $\vol(S) = \nnz(\Q_{:,\suppast})$ and $\intvol(S) = \nnz(\Q_{\suppast,\suppast})$, where $\newtarget{def:number_of_non_zeros}{\nnz}(\cdot)$ refers the number of non-zeros of a matrix, $\Q_{:,\suppast}$ refers to the columns of $\Q$ indexed by $\suppast$ and $\Q_{\suppast,\suppast}$ to the submatrix with entries $\Q_{i,j}$ for $i,j\in\suppast$. This is the formal definition of $\vol(\cdot)$ and $\intvol(\cdot)$ that we use when working with a general $M$-matrix $\Q$.  The complexity of our results depends on $\vol(\suppast)$ and $\intvol(\suppast)$.
\citet{fountoulakis2019variational} showed that for the $\ell_1$-regularized PageRank problem it is $\sum_{i\in \suppast} d_i \leq \frac{1}{\rho}$ and therefore $\vol(\suppast) \leq \frac{1}{\rho} + \sparsity$, where $\rho$ is the regularization parameter of the problem, see for example \eqref{eq:old_opt}.

\section[Personalized PageRank with l1-Regularization]{Personalized PageRank with $\ell_1$-Regularization}\label{sec:personalizedpagerank}

In this section, we introduce the PageRank problem that we study in this work, and we recall the variational formulation due to \citet{fountoulakis2019variational}. Let $\G = (\V, \edges)$ be a connected undirected graph with $\n$ vertices. We note that there are techniques to reduce an unconnected PageRank problem to a connected one, see for example \citet{eiron2004ranking}. Denote the adjacency matrix of $\G$ by $\newtarget{def:adjacency_matrix}{\A}$, that is, $\A_{i,j} = 1$ if $i\neigh j$ and $0$ otherwise. Let $\newtarget{def:diagonal_degree_matrix}{\D} \defi\operatorname{diag}(d_1, \dots, d_n)$ be the matrix with the degrees $\{d_i\}_{i=1}^{\n}$ in its diagonal. For $\alpha \in ]0, 1[$, consider the matrix
\begin{align}\label{eq:Q}
    \Q = \D^{-1/2} \left(\D - \frac{1-\alpha}{2} (\D + \A)\right)\D^{-1/2} = \alpha \I + \frac{1-\alpha}{2}\lapl \succ 0,
\end{align}
where $\newtarget{def:laplacian_matrix}{\lapl} \defi \I - \D^{-1/2}\A\D^{-1/2}$ is the symmetric normalized Laplacian matrix, which is known to be symmetric and satisfies $0 \prec \lapl \preccurlyeq  2 \I$ \citep{butler2006spectral}, hence the positive definiteness of $\Q$. In fact, by construction,
$0 \prec \alpha \I  \preccurlyeq \Q \preccurlyeq  \L \I$, for $\L = 1$. Note that $\Q_{i,j} \leq 0$ for $i\neq j$, so indeed $\Q$ is a positive definite $M$-matrix, which is what our algorithms require.

Next, given a distribution $\newtarget{def:personalized_distribution}{\ss} \in \simplex{\n}$ over the nodes of the graph $\G$,  called teleportation distribution, the personalized PageRank problem consists of optimizing the objective
$
    \newtarget{def:function_f_PageRank_objective}{\f}(\xx) \defi \frac{1}{2} \innp{ \xx , \Q \xx } - \alpha \innp{ \ss, \D^{-1/2} \xx }.
$
It holds that
$
    \nabla \f(\xx) = \Q \xx - \alpha \D^{-1/2}\ss,
$
$\nabla^2\f(\xx) = \Q$, and, thus, 
$\f$ is $\alpha$-strongly convex and $\L$-smooth. For $\newtarget{def:weight_in_l1_penalty}{\rho} > 0$, we are interested in the optimization of the $\ell_1$-regularized problem
\begin{align}\label{eq:old_opt}
    \min_{\xx \in \R^{\n}} \f(\xx) + \alpha\rho \norm{\D^{1/2} \xx}_1.
\end{align}
Solving \eqref{eq:old_opt} yields the same guarantees as \appr{}, see \citet{fountoulakis2019variational}. The advantage of the variational formulation \eqref{eq:old_opt} is that it allows to address the problem from an optimization perspective, as opposed to the algorithmic one of \appr{}, see \citet{andersen2006local}. 
Due to the strong convexity of the objective, \eqref{eq:old_opt} has a unique minimizer $\xast$. \citet{fountoulakis2019variational} proved that $\xast \geq \zeroterm$, which implies the following optimality conditions for \eqref{eq:old_opt} and $i\in [n]$:
\begin{align}\label{eq:old_optimality_conditions}
        \nabla_i \f(\xast) =  -\alpha\rho d_i^{1/2} \ \ \text{ if } \ \ \xast[i] > 0 \quad\quad\text{ and }\quad\quad \nabla_i \f(\xast)\in [-\alpha\rho d_i^{1/2}, 0]\ \ \text{ if } \ \ \xast[i] = 0.
\end{align}
Letting
\begin{align}
\label{eq:g}
    \g(\xx) \defi \f(\xx) + \alpha\rho \innp{ \oneterm,  \D^{1/2}\xx}  = \frac{1}{2}\innp{\xx, \Q\xx} + \alpha \innp{ \ss + \rho\oneterm, \D^{-1/2} \xx },
\end{align}
the optimality conditions for
$
     \argmin_{\xx\in\R^n_{\geq 0}}\g(\xx)
$
are equivalent to \eqref{eq:old_optimality_conditions}, that is, to the optimality conditions of Problem \eqref{eq:old_opt} and we have
\begin{align}\label{eq:opt_equivalent}
    \min_{\xx\in \R^{\n}}\f(\xx) + \alpha\rho \norm{\D^{1/2} \xx}_1
    = \min_{\xx\in\R^n_{\geq 0}}\g(\xx).
\end{align}
Put differently, at $\xast=\argmin_{\xx\in\R^n_{\geq 0}}\g(\xx)$, the following optimality conditions hold for $i\in[n]$:
\begin{align}\label{eq:new_optimality_conditions}
    \quad\quad \nabla_i \g (\xast) = 0\ \ \text{ if } \ \ \xast[i] > 0 \quad\quad \text{ and }\quad\quad \nabla_i \g (\xast) \in [0, \alpha\rho d_i^{1/2} ], \ \  \text{ if } \ \ \xast[i] = 0.
\end{align}
The algorithms presented in this work apply in particular to the minimization of $\g$ defined in \eqref{eq:g}.

\subsection[Projected Gradient Descent (PGD)]{Projected Gradient Descent (\pgd{})}\label{sec:pgd}

\begin{algorithm}
    \caption{Projected gradient descent (\pgd{})}\label{alg:pgd}
\begin{algorithmic}[1]
    \REQUIRE Closed and convex set $ C\subseteq \R^{\n}$, initial point $\xt[0] \in  C$, $f\colon  C \to \R$ an $\alpha$-strongly convex and $\L$-smooth function, and $\T\in \N$.
    \ENSURE $\xt[\T]\in  C$.
    \vspace{0.1cm}
    \hrule
    \vspace{0.1cm}
    \FOR {$t= 0, 1, \ldots, \T-1$}
       \State $\xt[t+1] \gets \proj{C}\left(\xt[t] - \frac{1}{\L}\nabla f(\xt[t])\right)$
    \ENDFOR
\end{algorithmic}
\end{algorithm}

\citet{fountoulakis2019variational} tackled Problem \eqref{eq:old_opt} by applying \ista{} to it, initialized at $\zeroterm$, and they showed that each iterate $\x[t]$ of the algorithm satisfies $\x[t] \geq \zeroterm$. Given $\x[t-1]$, the update rule of \ista{} defines the next iterate as
$
  \x[t] \defi \argmin_{\xx \in\R^{\n}} \rho\alpha\norm{\D^{1/2}\xx}_1 + \frac{1}{2}\norm{ \xx- (\x[t-1]-\nabla_i \f (\x[t-1]))}_2^2 = \argmin_{\xx\in\Rp^{\n}} \frac{1}{2}\norm{\xx- (\x[t]-\nabla \g (\x[t]))}^2,
$
where the equality follows directly by checking each coordinate, since the problems are separable. We note that the right hand side is the optimization problem that defines \pgd{} for $\g$ in $\Rp^{\n}$ . We present projected gradient descent (\newtarget{def:acronym_projected_gradient_descent}{\pgd{}}) in \cref{alg:pgd}, which will be useful to our analysis. None of our algorithms for addressing \eqref{eq:new_optimality_conditions} run \pgd{} as a subroutine.
The application of \pgd{} to the set $C\subseteq \R^{\n}$, initial point $\xt[0]\in  C$, objective $f\colon C \to \R$, and number of iterations $\T\in\N$ is denoted by $\xt[\T] = \pgd{}( C, \xt[0], f, \T)$.

\begin{fact}[Convergence rate of \pgd{}]\label{thm:pgd}
Let $ C\subseteq \R^{\n}$ be a closed convex set, $\xt[0] \in  C$, and $f\colon C\to\R$ an $\alpha$-strongly convex and $\L$-smooth function with minimizer at $\xast$. Then, for the iterates of \cref{alg:pgd}, it holds that
$
    \norm{\x[t] - \xxast}_2^2\leq \left(1 - \frac{1}{\kappa}\right)^t \norm{\x[0] - \xxast}_2^2,
$
where $\kappa\defi\frac{\L}{\alpha}$.  See \citet[Theorem~2.2.8]{nesterov1998introductory} for a proof. 
\end{fact}

\subsection{Geometrical Understanding of the Problem Setting}\label{sec:geometry}

\citet{fountoulakis2019variational} proved that for their method, the iterates $\x[t]$ never decrease coordinate-wise, and they concluded $\xast \in \Rp^{\n}$ as a consequence of this fact and the convergence guarantees of \ista{}: $\zeroterm \leq \x[1] \leq \dots \leq \x[t] \leq \x[t+1]\to \xast$. We generalize this result proven for the iterates of \ista{} to several geometric statements on the problem. This result holds in a more general setting, namely a quadratic with a positive-definite $M$-matrix as Hessian. The proof illustrates the geometry of the problem and we include it below. For any point $\xx$ such that $x_i = 0$ if $i\not\in\suppast$ and $\nabla_{i} \g(\xx) \leq 0$ if $i\in\suppast$, we have that $\xx \leq \xast$, among other things.

\begin{proposition}\label{proposition:pgd_helper}
    Let $\g$ be as in \eqref{eq:g} and let $S \subseteq [\n]$ be a set of indices such that we have a point $\xt[0]\in\Rp^{\n}$ with $\xt[0][i] = 0$ if $i\not\in S$ and $\nabla_i \g(\xt[0])\leq 0$ if $i\in S$. Let $C\defi\spann{\{\canonical[i] \mid i \in S \}} \cap \Rp^{\n}$, $\xtast[C] \defi \argmin_{\xx\in C} \g(\xx)$ and $\xast \defi \argmin_{\xx\in\Rp^{\n}} \g(\xx)$. Then:
\begin{enumerate}
    \item \label{property:pgd_helper_monotone} It holds that $\xt[0]\leq \xtast[C]$ and $\nabla_i \g(\xtast[C]) = 0$ for all $i \in S$.
    \item \label{property:pgd_helper_positivity}  If for $i\in S$, we have $x^{(0)}_i > 0$ or $\nabla_i \g(\xt[0]) < 0$, then $\xtast[C][i] > 0$.
    \item \label{property:pgd_helper_subset} If $\xtast[C][i] > 0$ for all $i \in S$, we have $\xtast[C] \leq \xast$ and therefore $S \subseteq \suppast$.  
\end{enumerate}
\end{proposition}

\begin{proof}
First, by definition of $C$, for all $\xx \in C$, we have $x_i = 0$ if $i\not\in S$. 
    Let $\{\xt[t]\}_{t=0}^\infty$ be the sequence of iterates created by \pgd{}$(C,\xt[0], \g, \cdot)$ when the algorithm is run for infinitely many iterations. We first prove that for all $t \geq 0$ and for all $i \in S$, we have $\nabla_i \g(\xt[t])\leq 0$. It holds for $t=0$ by assumption. If we assume it holds for some $t\geq 0$, then we have 
    \begin{equation}\label{eq:aux:pgd_update_computation}
    x^{(t+1)}_i =   x^{(t)}_i - \frac{1}{\L}\nabla_i \g (\xt[t]) \geq x^{(t)}_i
    \end{equation}
    for all $i \in S$, that is, the points do not decrease coordinatewise. Let the function $\newtarget{def:function_g_restricted_to_subspace}{\gbar}$ be $\g$ restricted to $\spann{\{\canonical[i] \mid i \in S\}}$ and note $\nabla_i \g(\xx) = \nabla_i \gbar(\xx)$ for $i \in S$. The function $\gbar$ is a quadratic with Hessian $\Q_{S,S}$, that is, it is formed by $\Q_{i,j}$ for  $i,j \in S$. Quadratics have affine gradients and so we have by \eqref{eq:aux:pgd_update_computation} that $\nabla \gbar(\xt[t+1]) = \nabla \gbar(\xt[t]) - \frac{1}{\L}\Q_{S,S}\nabla \gbar(\xt[t]) \leq 0$, where the last inequality is due to the assumption $\nabla \gbar(\xt[t]) \leq 0$, and $(\I-\frac{1}{\L}\Q_{S,S})_{i,j} \geq 0$ for all $i, j \in S$. The latter holds because for $i, j \in S$, $i\neq j$, we have $\Q_{i,j} \leq 0$ and due to smoothness, it is $\Q_{i, i} = e_i^\transp \Q e_i \leq \L$.
    Thus, by induction, for all $t\in \N$ and $i \in S$, we have $\nabla_ig(\xt[t])\leq 0$. This has two consequences. Firstly, $\xt[0]\leq \xt[1] \leq \dots $, and so $\xt[0] \leq \xtast[C]$ since the iterates of \pgd{} converge to $\xtast[C]$, by \cref{thm:pgd}. Secondly, using the limit and continuity of $\nabla \g(\cdot)$, it is $\nabla_i \g(\xtast[C]) \leq 0$ for $i\in S$. This fact and the optimality of $\xtast[C]$ imply $\nabla_ig(\xtast[C]) = 0$ for all $i\in S$, proving the first statement.

     For the second stament, fix $i \in S$. Note that by the assumption and the update rule $x^{(t+1)}_i =   x^{(t)}_i - \frac{1}{\L}\nabla_i \g (\xt[t])$, it holds that $x^{(1)}_i > 0 $, and thus, since $\xtast[C][i] \geq \xt[1][i]$, we have $\xtast[C][i] > 0$.

     For the third statement, we sequentially apply the first one to obtain optimizers in increasing subspaces, until we reach $\xast$, while showing they do not decrease coordinatewise. Suppose that  $\xtast[C][i] > 0$ for all $i\in S$. If $\xtast[C]= \xast$, the statement holds. Thus, we assume that $\xtast[C]\neq \xast$. 
     In that case, for $k \in \N$, define the optimizer $\yy^{(*, k)} \defi \argmin_{\yy\in B^{(k-1)}}\g(\yy)$ with respect to the set $B^{(k-1)} \defi \spann{\{\canonical[i] \mid i \in R^{(k-1)}\}} \cap \R^n_{\geq 0}$, where $R^{(k-1)} \defi R^{(k-2)} \cup N^{(k-1)}$ for $k>0$ and $R^{(-1)} \defi S$, and where $N^{(k-1)}\defi \{i\in[\n] \mid y^{(*,k-1)}_i = 0, \nabla_i \g(\yy^{(*,k-1)}) < 0\}$.
    By \cref{property:pgd_helper_monotone}, it holds that $\xtast[C] = \yy^{(*,0)} \leq \ldots \leq \yy^{(*,k)}$ and $\nabla_ig(\yy^{(*,k)}) = 0$ for all $i\in R^{(k-1)}$ and $k\in\N$. 
    Let $K \in \N$ denote the first iteration for which $R^{(K)} = R^{(K-1)}$, or, equivalently $N^{(K)} = \emptyset$. The existence of such a $K$ is guaranteed because otherwise $R^{(k)}\subset R^{(k+1)}$ for all $k\in \N$, but necessarily it is $|R^{(k)}| \leq \n$. 
    Thus, $\nabla_ig(\yy^{(*,K)}) = 0$ for all $i\in R^{(K-1)}$ and $\nabla_ig(\yy^{(*,K)}) \geq 0$ for all $i\not\in R^{(K-1)}$. In summary, $\yy^{(*, K)} $ satisfies the optimality conditions of the problem $\min_{\xx\in\Rp^{\n}} \g(\xx)$, implying that $\yy^{(*,K)} = \xast$. Since $S= R^{(-1)} \subseteq R^{(K)}$, \cref{property:pgd_helper_subset} holds.
\end{proof}

\subsection{Algorithmic Intuition}\label{sec:intuition}

In this section, we present the high-level idea of our algorithms for addressing \eqref{eq:opt_equivalent}. 
The core idea behind them is to start with the set of known good indices $ \Sinit=\emptyset$ and iteratively expand it, $\Sinit \subsetneq \St[0] \subsetneq \ldots \subsetneq \St[\T]$, until we have $\St[\T]=\suppast$ or we find an $\newtarget{def:accuracy_epsilon}{\epsilon}$-minimizer of \eqref{eq:opt_equivalent}. 
For $t\in\{0,1,\ldots, \T\}$, to determine elements $i\in \suppast\setminus \St[t-1]$, we let
\begin{align}\label{eq:algs_solve_this}
    \newtarget{def:optimizer_in_subspace}{\xtast[t]} =\argmin_{\xx\in \Ct[t-1]} \g(\xx),
\end{align}
where $\Ct[t-1] \defi \spann{\{\canonical[i] \mid i \in \St[t-1]\}} \cap \R^n_{\geq 0}$. By an argument following \cref{proposition:pgd_helper} that we will detail later, $\nabla_ig(\xtast[t]) < 0$ for at least one $i\in \suppast\setminus \St[t-1]$ and $\nabla_jg(\xtast[t]) \geq 0$ for all $j \not\in \suppast$. This observation motivates the following procedure: At iteration $t\in\{0,1,\ldots, \T-1\}$, construct $\xtast[t]$, check if $N^{(t)} = \{i\in[\n]\mid \nabla_i \g(\xtast[t]) < 0\}$ is not empty, and, in such a case, set $\St[t] \subseteq \St[t-1] \cup N^{(t)}$ and repeat the procedure. Should it ever happen that $N^{(t)}=\emptyset$, then we have $\xtast[t] = \xast$, that is, we found the optimal solution to \eqref{eq:opt_equivalent} and the algorithm can be terminated.
When using conjugate directions as the optimization algorithm for constructing \eqref{eq:algs_solve_this}, and when only incorporating good coordinates one by one, we obtain \cref{alg:sparse_conjugate_directions} (\newtarget{def:acronym_conjugate_directions_pagerank}{\cdappr{}}), see \cref{sec:cdappr}. For our second algorithm, \cref{alg:sparse_acceleration} (\aspr{}), we use accelerated projected gradient descent to construct only an approximation of \eqref{eq:algs_solve_this} and we show that this method still allows us to proceed. We discuss the subtleties arising from using an approximation algorithm in \cref{sec:apgdappr}.

\section{Conjugate Directions for PageRank}\label{sec:cdappr}
{ 
\input{definitions_cd.tex}

\begin{algorithm}
    \caption{Conjugate directions PageRank algorithm (\cdappr{})} \label{alg:sparse_conjugate_directions}
\begin{algorithmic}[1]
    \REQUIRE Quadratic function $g:\R^{\n}\to \R$ with Hessian $\Q \succ 0$ being a symmetric $M$-matrix. The $\ell_1$-regularized PageRank problem corresponds to choosing $\g$ as in \eqref{eq:g}.
    \ENSURE $\xt[\T] =\argmin_{\xx\in\Rp^{\n}} \g(\xx)$, where $\newtarget{def:final_iteration_T_of_conjugate_directions}{\T} \in \N$ is the first iteration for which $N^{(\T)} = \emptyset$. 
    \vspace{0.1cm}
    \hrule
    \vspace{0.1cm}
    \State $t\gets 0$
    \State $\xt[t]\gets \zeroterm$
    \State $N^{(t)} \gets \left\{i\in[\n] \mid \nabla_i \g(\xt[t]) < 0\right\}$
    \WHILE{$N^{(t)} \neq \emptyset$}
      \State $i^{(t)} \in N^{(t)}$
      \State $\newtarget{def:initial_basis_ut_in_algorithm}{\ut[t]} \gets \nabla_{i^{(t)}}\g(\xt[t])\cdot \canonical[i^{(t)}]$
      \State  $\beta_k^{(t)} \gets \nabla_{i^{(t)}}\g(\xt[t])\innp{\Q_{i^{(t)},:}, \dtbar[k]} $ for all $k = 0, \dots, t-1$ \Comment{equal to $-\frac{\innp{\ut[t],\Q\dt[k]}}{\innp{ \dt[k], \Q\dt[k]}}$}
      \State $\newtarget{def:directions_dt_in_conjugate_directions}{\dt[t]} \gets \ut[t] + \sum_{k=0}^{t-1} \beta_k^{(t)} \dt[k]$\Comment{store this sparse vector}
      \State $\newtarget{def:stored_normalized_directions_dt_in_conjugate_directions}{\dtbar[t]} \gets \frac{ \dt[t]}{\innp{ \dt[t], \Q \dt[t]}} $ \Comment{store $\innp{ \dt[t], \Q \dt[t]}$}
      \State $\eta^{(t)}\gets -\innp{\nabla \g(\xt[t]), \dtbar[t]}$ 
      \State $\newtarget{def:iterate_xt_of_CDPR}{\xt[t+1]}\gets \xt[t] + \eta^{(t)}\dt[t]$
      \State $N^{(t+1)} \gets \left\{i\in[\n] \mid \nabla_i \g(\xt[t+1]) < 0\right\}$
      \State $t\gets t +1$
  \ENDWHILE
\end{algorithmic}
\end{algorithm}

With the geometric properties of the problem we established in \cref{sec:personalizedpagerank}, we are ready to introduce the conjugate directions PageRank algorithm (\cdappr{}) \cref{alg:sparse_conjugate_directions}, a \emph{conjugate-directions}-based approach for addressing \eqref{eq:opt_equivalent}, which outperforms the \ista{}-solver due to \citet{fountoulakis2019variational} in certain parameter regimes. \cdappr{} is based on the algorithmic blueprint outlined in \cref{sec:intuition} and constructs $\xtast[\T]$ as in \eqref{eq:algs_solve_this} using conjugate directions. As we will prove formally, it is $\zeroterm \leq \xtast[t]$ for all $t\in\{0,1,\ldots, \T\}$, allowing us to solve the constrained problem \eqref{eq:algs_solve_this} by dropping the non-negativity constraints and using the method of conjugate directions. This is an important point, since this method is designed for affine spaces only and, to the best of our knowledge, cannot deal with other constraints. Conjugate directions are an attractive mechanism for finding \eqref{eq:algs_solve_this}, as it allows to exploit the sparsity of the solution, is exact, and does not rely on the strong convexity of the objective, leading to a time complexity independent of $\alpha$. Note that, even though we may learn about several new good coordinates at the end of an iteration, in order to maintain the invariants required for our \cdappr{}, we can add at most one new coordinate to $\St[t]$ at a time. This algorithm requires more memory than the \ista{}-solver of \citet{fountoulakis2019variational} and \aspr{}, which is due to storing an increasing $\Q$-orthogonal basis that is required to perform exact optimization over $\Ct[t]$ by performing Gram-Schmidt with respect to $\Q$.

\cref{alg:sparse_conjugate_directions} works in the following way. Initialize with $\newtarget{def:initial_set_of_known_good_coordinates_CDPR}{\Sinit} \defi \emptyset$, and $\xtast[0]\defi \zeroterm$. For $t \in \{0, 1,\ldots, \T\}$, let the set of known good coordinates be $\newtarget{def:set_of_known_good_coordinates_CDPR}{\St[t]}\defi \St[t-1] \cup \{i^{(t)}\}$, and define $\newtarget{def:span_of_known_good_coordinates_in_Rp_CDPR}{\Ct[t]} \defi \spann{\{\canonical[i] \mid i \in \St[t] \}} \cap \R^n_{\geq 0}$, and 
$\newtarget{def:optimizer_in_subspace_for_CDPR}{\xtast[t]} \defi\argmin_{\xx\in C^{(t-1)}} \g(\xx)$. At each iteration $t \in \{0, 1, \ldots, \T-1\}$, we start at $\xtast[t] \geq 0$, for which it holds that $\nabla_i \g(\xtast[t]) = 0$ for $i\in\St[t-1]$ and there exists at least one $i^{(t)} \not \in \St[t-1]$ such that $\nabla_{i^{(t)}} \g(\xtast[t]) < 0$ unless we are already at the optimal solution, that is, $\xtast[t-1] = \xast$. We arbitrarily select one such index, and then perform Gram-Schmidt with respect to $\Q$ in order to obtain $\dt[t]$ that is $\Q$-orthogonal to all $\dt[k]$ for $k < t$. Next,  one can see that the optimizer $\xt[t+1]$ along the line $\xt[t] +\eta^{(t)} \dt[t]$ results in the optimizer for the subspace $\spann{\{\canonical[i] \mid i \in \St[t] \}}$, which is $\xtast[t+1] \geq 0$.
After $\sparsity$ iterations, we obtain $\xast$.  We formalize and prove the claims of the overview below.

\begin{theorem}\linktoproof{thm:cd_approach}\label{thm:cd_approach}
    For all $t\in \{0, 1, \ldots, \T\}$ and $k\in \{0,1,\ldots, t-1\}$, the following properties are satisfied for \cref{alg:sparse_conjugate_directions}:
    \begin{enumerate}
    \item \label{property:cd_approach_orthogonality} It holds that $\innp{ \dt[t], \Q \dt[k]} = 0$.
    \item \label{property:cd_approach_orthogonality_gradient} We have that
    $
            \innp{\nabla \g (\xt[t]), \dt[k] }= 0
    $
    and
     $\nabla_i \g(\xt[t]) = 0$ for all $i\in \St[t-1]$.
    \item  \label{property:cd_approach_x_monotone}
    It is $\xx_i^{(t)} > 0$ for all $i\in \St[t-1]$, and $\zeroterm = \xt[0]=\xtast[0]  \leq \xt[1]=\xtast[1]\leq \ldots \leq \xt[\T]= \xtast[\T]$. 
    \item \label{property:cd_approach_optimality} It holds that $\xt[\T] = \xast$.
    \end{enumerate}
\end{theorem}

Unlike our next algorithm, \aspr{}, the time complexity of \cref{alg:sparse_conjugate_directions} does not depend on $\alpha$, $\L$, or $\epsilon$, and we optimize exactly. We detail the computational complexities of our algorithm below.

\begin{theorem}[Computational complexities]\label{thm:cd_approach_complexity}\linktoproof{thm:cd_approach_complexity}
    The time complexity of \cref{alg:sparse_conjugate_directions}
    is
    $\bigo{| \suppast |^3 + | \suppast |  \vol(\suppast)}$ and its space complexity is $\bigo{| \suppast |^2}$.
\end{theorem}

\input{outro_cd.tex}
}

\section{Accelerated Sparse PageRank}\label{sec:apgdappr}
{ 
\input{definitions_aspr.tex}

\begin{algorithm}
    \caption{Accelerated projected gradient descent (\apgd{})} \label{alg:apgd}
\begin{algorithmic}[1]
    \REQUIRE  Closed and convex set $ C\subseteq \R^{\n}$, initial point $\xt[0] \in  C$, $f\colon  C \to \R$ an $\alpha$-strongly convex and $\L$-smooth function, condition number $\kappa \defi \L/\alpha$, and $T\in \N$.
    \ENSURE $\y[T]\in  C$.
    \vspace{0.1cm}
    \hrule
    \vspace{0.1cm}
    \State $\z[0] \gets\y[0] \gets \x[0]$; \ \ $A_0 \gets 0$; \ \ $a_0\gets 1$
    \FOR{$t= 0, 1, \ldots, T-1$}
      \State $A_{t+1} \gets A_{t} + a_{t}$ \Comment{equal to $A_t (\frac{2\kappa}{2\kappa +1 - \sqrt{1+4\kappa}}) \geq A_t(1-\frac{1}{2\sqrt{\kappa}})^{-1}$ if $t\geq 1$}
      \State $\x[t+1] \gets \frac{A_{t}}{A_{t+1}}\y[t] + \frac{a_{t}}{A_{t+1}} \z[t]$
      \State $\z[t+1] \gets \proj{C}\left( \frac{\kappa-1+ A_{t}}{\kappa-1+A_{t+1}} \z[t] +  \frac{a_{t}}{\kappa-1+ A_{t+1}} \left(\x[t+1] - \frac{1}{\alpha}\nabla f(\x[t+1])\right) \right)$
      \State $\y[t+1] \gets \frac{A_{t}}{A_{t+1}}\y[t] + \frac{a_{t}}{A_{t+1}} \z[t+1]$ 
      \State $a_{t+1} \gets A_{t+1}(\frac{2\kappa}{2\kappa +1 - \sqrt{1+4\kappa}}-1)$
      \ENDFOR
\end{algorithmic}
\end{algorithm}

In this section, we introduce  the accelerated sparse PageRank algorithm (\newtarget{def:acronym_accelerated_sparse_pagerank}{\aspr{}}) in \cref{alg:sparse_acceleration}, which is an approach based on accelerated projected gradient descent (\newtarget{def:acronym_accelerated_gradient_descent}{\apgd{}}) for addressing \eqref{eq:opt_equivalent}. 
Let $\xtast[0] = \zeroterm$, let $\Sinit = \emptyset$, and for $t\in[\T]$, let $\xtast[t] =\argmin_{\xx\in \Ct[t-1]} \g(\xx)$.
We now explain the necessary modifications to the exact algorithm outlined in \cref{sec:intuition} such that an approximate solver of \eqref{eq:algs_solve_this} can be incorporated. 
First, we recall the convergence of accelerated projected gradient descent (\apgd{}) \citep{nesterov1998introductory} in \cref{alg:apgd}, which is used as a subroutine in \cref{alg:sparse_acceleration}.
\apgd{} applied to the set $C\subseteq \R^{\n}$, initial point $\xt[0]\in  C$, objective $f\colon C \to \R$, and number of iterations $\T\in\N$ is denoted by $\xx\gets\apgd{}( C, \xt[0], f, \T)$. For strongly convex objectives, \apgd{} enjoys the following convergence rate.

\begin{proposition}[Convergence rate of \apgd{}]\linktoproof{prop:apgd}\label{prop:apgd} 
Let $ C\subseteq \R^{\n}$ be a closed convex set, $\xt[0] \in  C$, and $f\colon C\to\R$ an $\alpha$-strongly convex and $\L$-smooth function with minimizer $\xxast$. Then, for the iterates of \cref{alg:apgd}, it holds that
$
    f(\y[t]) - f(\xxast) \leq (1 - \frac{1}{2\sqrt{\kappa}})^{t-1} \frac{(\L-\alpha)\norm{\xt[0]-\xxast}^2}{2},
$
for $\kappa\defi\frac{\L}{\alpha}$.
    We thus obtain an $\epsilon$-minimizer in $\T = 1+\ceil{2\sqrt{\kappa}\log(\frac{(\L-\alpha)\norm{\x[0]-\xxast}^2}{2\epsilon})} $
    $\leq 1+\ceil{2\sqrt{\kappa}\log(\frac{(\L-\alpha)\norm{\nabla f(\xt[0])}_2^2}{2\epsilon\alpha^2})}$ iterations.
\end{proposition}

As in the \cref{alg:sparse_conjugate_directions} in the previous section, our \cref{alg:sparse_acceleration} constructs a sequence of subsets $\St[t]$ of the support of $\xast$. 
In contrast to \cdappr{}, \cref{alg:sparse_acceleration} does not compute $\xtast[t+1]$ for $t\in \{0,1,\ldots, \T-1\}$ exactly, but instead employs \apgd{} as a subroutine to construct a point $\xtbar[t+1]$ that is close enough to $\xtast[t+1]$, and then reduces all positive entries of $\xtbar[t+1]$ slightly, obtaining $\xt[t+1] \leq \xtast[t+1]$. 
The following \cref{lemma:A_positive} establishes that if a coordinate of a point is decreased, the gradient of $\g$ at all other coordinates does not decrease, implying that for all points $\xx\in\R^n_{\geq 0}$ satisfying $\xx\leq \xtast[t+1]$, no bad coordinate has a negative gradient.

\begin{algorithm}
    \caption{Accelerated sparse PageRank algorithm (\aspr{})} \label{alg:sparse_acceleration}
\begin{algorithmic}[1]
    \REQUIRE Quadratic function $g:\R^{\n}\to \R$ with Hessian $\Q \succ 0$ being a symmetric $M$-matrix, accuracy $\epsilon > 0$. The $\ell_1$-regularized PageRank problem corresponds to choosing $\g$ as in \eqref{eq:g}.
    
    \ENSURE $\xt[\T]$, where $\newtarget{def:final_iteration_T_of_ASPR}{\T} \in \N$ is the first iteration for which $\St[\T] = \St[\T-1]$.
    \vspace{0.1cm}
    \hrule
    \vspace{0.1cm}
    \State $t\gets 0$
    \State $\xt[0] \gets \zeroterm$
    \State $\St[t] \gets \{i\in[\n] \mid \nabla_i \g(\xt[t]) < 0\}$
    \WHILE{$\St[t] \neq \St[t-1]$}
        \State $\newtarget{def:retraction_parameter_delta_t}{\delta_t} \gets \sqrt{\frac{\epsilon\alpha}{(1+\card{\St[t]})\L^2}}$
        \State $\newtarget{def:accuracy_parameter_of_APGD_subproblem}{\hatepsilon}_t \gets \frac{\delta_t^2 \alpha }{2} = \frac{\epsilon\alpha^2}{2(1+\card{\St[t]})\L^2}$
        \State $\newtarget{def:span_of_known_good_coordinates_in_Rp_ASPR}{\Ct[t]} \gets \spann{\{\canonical[i] \mid i \in \St[t]\} } \cap \R^n_{\geq 0}$
        \State $\newtarget{def:iterate_before_pulling_towards_zero}{\xtbar[t+1]} \gets \apgd{}\left(\Ct[t], \xt[t], \g, 1+\Big\lceil 2\sqrt{\kappa}\log\left(\frac{(\L-\alpha)\norm{\nabla_{\St[t]} \g(\xt[t])}_2^2}{2\hatepsilon_t\alpha^2}\right)\Big\rceil\right)$ 
        \State $\newtarget{def:iterate_xt_of_ASPR}{\xt[t+1]} \gets \max\{\zeroterm, \xtbar[t+1]-\delta_t \oneterm\}$ \Comment{coordinatewise $\max$, only needed for $i\in \St[t]$}
        \State $\newtarget{def:set_of_known_good_coordinates_ASPR}{\St[t+1]} \gets  \St[t] \cup\{i\in[\n] \mid \nabla_i \g(\xt[t+1]) < 0\}$ \label{line:expanding_S_t}
        \State $t\gets t + 1$
    \ENDWHILE
\end{algorithmic}
\end{algorithm}

\begin{lemma}\linktoproof{lemma:A_positive}\label{lemma:A_positive}
    Let $\xx \in \R^{\n}$, and let $\yy = \xx - \epsilon  \canonical[i]$, for some $\epsilon > 0$, $i\in[\n]$. Then, for all $j\in [\n]\setminus\{i\}$, it holds that $\nabla_j \g(\yy) \geq \nabla_j \g(\xx)$. If instead $\epsilon < 0$, then $\nabla_j \g(\yy) \leq \nabla_j \g(\xx)$.
\end{lemma}
The second part of \cref{lemma:A_positive} implies that we only have $\nabla_i \g(\xt[t+1]) < 0$ for coordinates $i$ for which $\nabla_i \g(\xtast[t+1]) < 0$, but it suggests that there could be none satisfying the former. To address this issue, \apgd{} is run to sufficient accuracy to guarantee that $\g(\xt[t+1]) - \g(\xtast[t+1]) \leq \frac{\epsilon \alpha}{\L}$. Then, we show that either $\g(\xt[t+1]) -\g(\xast) \leq \epsilon$ or one step of \pgd{} from $\xt[t+1]$ would make more progress than what we can do in the current space $\Ct[t]$, of which $\xtast[t+1]$ is minimizer, and so the gradient contains a negative entry. All such entries are good coordinates $i\in\suppast\setminus \St[t]$, similarly to what we had at $\xtast[t+1]$ in \cdappr{}. We note that unlike for \cdappr{}, this time we can incorporate all of these coordinates at once to the algorithm. In \cref{prop:apgd_approach} below, we address all these challenges associated with computing $\xt[t+1]$ in \cref{alg:sparse_acceleration} in lieu of $\xtast[t+1]$, and we prove that indeed \cref{alg:sparse_acceleration} finds an $\epsilon$-minimizer of $\g$, while all the iterates are sparse, if the solution $\xast$ is sparse.

\begin{theorem}\linktoproof{prop:apgd_approach}\label{prop:apgd_approach}
    Let $\newtarget{def:initial_set_of_known_good_coordinates_ASPR}{\Sinit} \defi \emptyset$, $\xtast[-1]\defi \zeroterm$, $\xtast[0] \defi \zeroterm$, and define $\newtarget{def:optimizer_in_subspace_for_ASPR}{\xtast[t]} \defi \argmin_{\xx\in \Ct[t-1]} \g(\xx)$ for $t \in[\T]$, where $\Ct[t-1]$ is defined in \cref{alg:sparse_acceleration}. For all $t \in \{0, 1, \dots, \T\}$, the following properties are satisfied for \cref{alg:sparse_acceleration}:
    \begin{enumerate}
        \item \label{property:apgd_approach_positivity_grad} It holds $\xtast[t][i] > 0$ if and only if $i \in \St[t-1]$. We also have $\nabla_i \g(\xtast[t]) = 0$ if $i\in \St[t-1]$.
        \item \label{property:apgd_approach_x_monotone} It is $ \xt[t] \leq \xtast[t] \leq \xast$ and $\xtast[t-1] \leq \xtast[t]$.
        \item \label{property:apgd_approach_S_monotone} Our set of known good indices expands $\St[t-1] \subsetneq \St[t] \defi \St[t-1] \cup \left\{i\in [\n] \mid \nabla_i \g(\xt[t])<0 \right\} \subseteq \suppast$, or $\xt[t]$ is an $\epsilon$-minimizer of $\g$. In particular, $\g(\xt[\T]) - \g(\xast) \leq \epsilon$.
    \end{enumerate}
\end{theorem}

Note that by the previous theorem, we have the chain $\zeroterm = \xtast[0] \leq \xtast[ 1] \leq \ldots \leq \xtast[ \T] \leq \xast$ and $\Sinit \subsetneq \St[0] \subsetneq \ldots \subsetneq \St[\T-1] = \St[\T] \subseteq \suppast$. This implies that every iterate of \cref{alg:sparse_acceleration} only updates coordinates in $\suppast$. Thus, the final computational complexity of this accelerated method, specified below, depends on the sparsity of the solution and related quantities, answering the question posed by \citep{fountoulakis2022open} in the affirmative.
\begin{theorem}[Computational complexities]\linktoproof{thm:apgd_approach_complexity}\label{thm:apgd_approach_complexity}
    The time complexity of \cref{alg:sparse_acceleration}
    is
\begin{align*}
    \bigotildel{\sparsity\intvol(\suppast)\sqrt{\frac{\L}{\alpha}} + \sparsity\vol(\suppast)},
\end{align*}
    and its space complexity is $\bigo{\sparsity}$.
\end{theorem}

The question of \citet{fountoulakis2022open} suggested that one has to possibly trade-off lower dependence on the condition number for greater dependence on the sparsity. Surprisingly, the term $\sparsity \intvol(\suppast)$ multiplying the condition-number term can be smaller than the corresponding term $\vol(\suppast)$ of \ista{}, so in such a case the accelerated method also improves on the dependence on the sparsity, and it enjoys an overall lower running time if $\sparsity < \L/\alpha$, see \cref{sec:algorithmic_comparisons}.

\subsection[Variants of Algorithm~\ref{alg:sparse_acceleration}]{Variants of \cref{alg:sparse_acceleration}}
An attractive property of \cref{alg:sparse_acceleration} is that by performing minor modifications to it, we can exploit the geometry to stop the \apgd{} subroutine earlier, or we can naturally incorporate new lower bounds on the coordinates of $\xast$ into the algorithm.

\subsubsection[Early Termination of APGD in Algorithm \ref{alg:sparse_acceleration}]{Early Termination of \apgd{} in \aspr{}} 
We first present another lemma about the geometry of Problem \eqref{eq:opt_equivalent}.

\begin{lemma}\linktoproof{lemma:modification}\label{lemma:modification}
    Let $S$ be a set of indices such that  $\xtast[C] \defi \argmin_{\xx \in C} \g(\xx)$ satisfies $\xtast[C][j] > 0 $ if and only if $j \in S$, where $C \defi \spann{\{\canonical[j] \mid j \in S \}} \cap \Rp^{\n}$. Let $\xx\in\R^n_{\geq 0}$ be such that $x_j = 0$ if $j\not\in S$ and $\nabla_j \g(\xx) \leq 0$ if $j\in S$. Then, for any coordinate $i\not\in S$ such that $\nabla_i \g(\xx) < 0$, we have $i\in\suppast$. 
\end{lemma}

By Statement~\ref{property:apgd_approach_positivity_grad} in \cref{prop:apgd_approach}, we can apply \cref{lemma:modification} with $\St[t]$ in \cref{alg:sparse_acceleration}, for any $t\in\{0,1,\ldots, \T-1\}$. This motivates the following modification to \cref{alg:sparse_acceleration}: In \aspr{}, if we compute the full gradient at each iteration (or every few iterations) in the \apgd{} subroutine, then, for an iterate $\xx$ in \apgd{}, if we have $\nabla_i \g(\xx) \leq 0$ for all $i\in \St[t]$ and we observe some $j \not\in \St[t]$ such that $\nabla_j \g(\xx) < 0$, then we can stop the \apgd{} subroutine, incorporate all such coordinates to $\St[t+1]$ and continue with the next iteration of \cref{alg:sparse_acceleration}.  
This modification does not come without drawbacks, since we need to compute full gradients in order to discover new coordinates early, instead of just gradients restricted to $\St[t]$. Interestingly, we can show that if we were to compute one full gradient for each iteration of the \apgd{} subroutine, then the complexity of the conjugate directions method is no worse than the upper bound on the complexity for this variant of \cref{alg:sparse_acceleration}, in the regime in which we prefer to use these algorithms over the \ista{}-approach in \citet{fountoulakis2019variational}.  Indeed, the complexity of this variant is $\bigotilde{\sparsity\volast \sqrt{\L/\alpha}}$, and the complexity of \cref{alg:sparse_conjugate_directions}, which is $\bigo{\sparsity^3 + \sparsity \vol(\suppast)}$, can be upper bounded by $\bigo{\sparsity^2  \volast}$. If the complexity of the variant is better, up to constants and log factors, then we can exchange another $\sparsity$ term by $\sqrt{\L/\alpha}$ to conclude that this complexity is no better than the complexity of the \ista{} approach $\bigotilde{\volast \frac{\L}{\alpha}}$, up to constants and log factors. Nonetheless, one can always compute the full gradient only sporadically to discover new good coordinates earlier, and we expect the empirical performance of \cref{alg:sparse_acceleration} to improve by implementing this modification. In future work, we will extensively test our algorithms with this and other variants to assess their practical performance.

\subsubsection{Updating Constraints}\label{sec:update_constraints}

In \cref{alg:sparse_acceleration}, every time we observe $\nabla_i \g(\xx) \leq 0$ for all $i \in \St[t]$, whether for the iterates of the \apgd{} subroutine or for $\xt[t+1]$, we have by Statement \ref{property:pgd_helper_monotone} of \cref{proposition:pgd_helper} and Statement \ref{property:apgd_approach_x_monotone} of \cref{prop:apgd_approach} that $\xx \leq \xtast[t+1] \leq \xast$. Using this new lower bound on the coordinates of $\xast$, we can update  our constraints. If we initialize the constraints to $\bar{C} \gets \Rp^{\n}$, we can update them to $\bar{C}\gets\bar{C} \cap \{\yy \in \Rp^{\n} \ |\ \yy \geq \xx\}$ every time we find one such point $\xx$. This can help avoiding the momentum of \apgd{} taking us far unnecessarily. We note that these constraints are isomorphic to the positive orthant, and only require storing up to $\sparsity$ numbers.

\input{outro_aspr.tex}
} 

\section{Conclusion}

We successfully integrated acceleration of optimization techniques into the field of graph clustering, thereby answering the open question raised in \citet{fountoulakis2019variational}. Our results provide evidence of the efficacy of this approach, demonstrating that optimization-based algorithms can be effectively employed to address graph-based learning tasks with great efficiency at scale. This work holds the potential to inspire the development of new algorithms that leverage the power of advanced optimization techniques to tackle other graph-based challenges in a scalable manner.

\clearpage

\acks{
This research was partially funded by the Deutsche Forschungsgemeinschaft (DFG, German Research Foundation) under Germany's Excellence Strategy – The Berlin Mathematics Research Center MATH$^+$ (EXC-2046/1, project ID 390685689, BMS Stipend).
}

\printbibliography[heading=bibintoc]

\clearpage

\appendix

\section{Missing Proofs}
\begin{remark}
We recall that, as we pointed out in \cref{sec:pgd}, \ISTA{} on $\f$ is equivalent to projected gradient descent in $\Rp^{\n}$ on $\g$. \cref{proposition:pgd_helper} allows to quite simply recover the result in \citep{fountoulakis2019variational} about \ISTA{} initialized at $\zeroterm$ having iterates with support in $\suppast$. Moreover, our argument below applies to the optimization of the more general problem where $\Q$ is an arbitrary symmetric positive-definite $M$-matrix.

    Indeed, we satisfy the assumptions of \cref{proposition:pgd_helper} for initial point $\x[0]$ and set of indices $S \gets \{i\ |\ \nabla_ig(\x[0]) < 0\}$, which is non-empty unless $\x[0]$ is the solution $\xast$. Let the corresponding feasible set be $C\defi\mathspan(\{\canonical[i] \ | \ i \in S \}) \cap \Rp^{\n}$. Now, while the iterates of $\pgd{}(\Rp, \x[0], \g, \cdot)$ remain in $C$, that is, for $t$ such that $\x[t] \in C$, they behave exactly like $\pgd{}(C, \x[0], \g, \cdot)$ and so, we have the following invariant by the proof of the \cref{proposition:pgd_helper}: $\x[t] \leq \xtast[C]$ and $\nabla_{S} \g(\x[t]) \leq 0$ and $\x[t] \leq \x[t+1]$. If this algorithm leaves $C$ at step $t$, that is $\x[t] \in C$ and $\x[t+1] \not \in C$, we have $\nabla_{\supp(\x[t+1])} \g(\x[t]) \leq 0$, since the invariant guarantees $\nabla_{S} \g(\x[t]) \leq 0$ and if $i \in \supp(\x[t+1]\setminus S$, it must be $\nabla_i \g(\x[t]) < 0$ by definition of the \pgd{} update rule. In particular, we can apply \cref{proposition:pgd_helper} again with initial point $\x[t]$ and the larger set of indices $\supp(\x[t+1])$, and so on, proving that the invariant $\nabla_{\supp(\x[t])} \g(\x[t]) \leq 0$ and $\x[t] \leq \x[t+1]$ holds for all $t \geq 0$.  By the global convergence of $\pgd{}(\Rp, \x[0], \g, \cdot)$, we have $\x[t] \leq \xast$ for all $t \geq 0$, so it is always $\supp(\x[t])\subseteq \suppast$.
\end{remark}

In the rest of this  section, we present proofs not found in the main text.

{

\input{definitions_cd.tex}
\begin{proof}\linkofproof{thm:cd_approach}
    We prove the properties in order:
    \begin{enumerate}
        \item [\ref{property:cd_approach_orthogonality}.] For $t = 0$, the statement is trivial. Let $t\in \{0,1,\ldots, \T-1\}$ and assume that $\innp{ \dt[j], \Q \dt[k]} = 0$ for all $j,k\in \{0,1,\ldots, t\}$ such that $j>k$. Then,
        \begin{align*}
            \innp{ \dt[t+1], \Q \dt[k]} & {=}  \innp{ \ut[t+1] {+} \sum_{k=0}^{t} \beta_k^{(t+1)} \dt[k], \Q \dt[k]} = \innp{\ut[t+1], \Q\dt[k] } + \beta_k^{(t+1)} \innp{\dt[k], \Q\dt[k]} = 0,
        \end{align*}
        where the second and third equalities follow from the induction hypothesis and the definition of $\beta^{(t+1)}_k$, respectively.

        \item
        [\ref{property:cd_approach_orthogonality_gradient}.]
        By induction. 
        For $t=0$, there is nothing to prove. For some $t\in \{0,1,\ldots, \T-1\}$ suppose that  for all $k\in \{0,1,\ldots, t\}$, it holds that
    $
            \innp{\nabla \g (\xt[t]), \dt[k] }= 0.
    $
    Then, since  $\xt[t+1]= \xt[t] + \eta^{(t)}\dt[t]$, we have $\nabla \g(\xt[t+1]) = \nabla \g (\xt[t]) + \eta^{(t)} \Q \dt[t]$. Thus, by \cref{property:cd_approach_orthogonality} and the induction hypothesis, we have
        \begin{align*}
            \innp{\nabla \g (\xt[t+1]), \dt[k] } & = \innp{\nabla \g (\xt[t]), \dt[k] } + \eta^{(t)}\innp{ \Q \dt[t], \dt[k] } = 0
        \end{align*}
        for all $k < t$. By the definition of $\eta^{(t)}$, we also have $\innp{\nabla \g (\xt[t+1]), \dt[t] } = 0$. Thus, 
        \begin{align}\label{eq:orthogonality}
            \innp{\nabla \g (\xt[t+1]), \dt[k] }= 0 \qquad \text{for all $t\in \{0,1,\ldots, \T-1\}$ and $k\in \{0,1,\ldots, t\}$.}
        \end{align}
            Thus, since for $t\in\{0, 1,\ldots, \T\}$, $\spann{\{\dt[0], \dt[1], \ldots, \dt[t-1]\}} = \spann{\{\canonical[i] \mid \canonical[i] \in \St[t-1]\}}$, it holds that $\nabla_i \g(\xt[t]) = 0$ for all $i\in \St[t-1]$ by \eqref{eq:orthogonality}.
        
        \item [\ref{property:cd_approach_x_monotone}.]
        By induction. For $t=0$, it holds that $\xx_i^{(0)} > 0$ for all $i\in \Sinit = \emptyset$ and $\zeroterm = \xt[0]=\xtast[0]$.
        Suppose that the statement holds for some $t\in\{0,1,\ldots, \T-1\}$, that is,
        $\xx_i^{(t)} > 0$ for all $i\in \St[t-1]$ and
        $\zeroterm = \xt[0]=\xtast[0]  \leq \xt[1]=\xtast[1]\leq \ldots \leq \xt[t]= \xtast[t]$.
            By \cref{proposition:pgd_helper} applied to $\g$, $\St[t]$, and $\xt[t] = \xtast[t]$, we have that $\xtast[t+1][i] > 0$ for all $i\in \St[t]$ and $\nabla_i \g(\xtast[t+1]) = 0$ for all $i\in \St[t]$, that is, 
        $\xtast[t+1]=\argmin_{\xx \in \spann{\{\canonical[i] \mid i\in \St[t]\}}} \g(\xx)$.  
        By \cref{property:cd_approach_orthogonality_gradient}, $\nabla_ig(\xt[t+1]) = 0$ for all $i\in \St[t]$, that is,
        $\xt[t+1]=\argmin_{\xx \in \spann{\{\canonical[i] \mid i\in \St[t]\}}} \g(\xx)$. By strong convexity of $\g$ restricted to $\spann{\{\canonical[i] \mid i\in \St[t]\}}$, $\xtast[t+1]=\xt[t+1]$.
        \item [\ref{property:cd_approach_optimality}.] By \cref{property:cd_approach_x_monotone}, $ \zeroterm \leq \xt[\T]$, that is, $\xt[\T]$ is a feasible solution to the optimization problem \eqref{eq:opt_equivalent}. By \cref{property:cd_approach_orthogonality_gradient}, $\nabla_i \g(\xt[\T]) = 0$ for all $i\in \St[\T-1]$. Since $\T\in \N$ is the first iteration for which $N^{(\T)} = \left\{i\in[\n] \mid \nabla_i \g(\xt[\T]) < 0\right\} = \emptyset$, $\xt[\T]$ satisfies the optimality conditions \eqref{eq:new_optimality_conditions} and $\xt[\T] = \xast$.
    \end{enumerate}
\end{proof}

\begin{proof}\linkofproof{thm:cd_approach_complexity}
    We run \cref{alg:sparse_conjugate_directions} for $| \suppast |$ iterations. We summarize the costs of operations performed during one iteration $t\in\{0, 1, \ldots, \T\}$. The cost of computing $N^{(t+1)}$ is $\bigo{\vol(\suppast)}$. Note we do not need to store the gradient, at most we would store $\nabla_{\St[t+1]} \g(\xt[t+1])$, and this is not necessary. Note that the vectors $\xt[t+1]$ and $\dt[t]$ and $\dtbar[t]$ are sparse, their support is in $\suppast$. Thus, computing $\bar{\dt[t]}$ takes $\bigo{\intvol(\suppast)}$ and computing $\eta^{(t)}$ and $\xt[t+1]$ takes $\bigo{\sparsity}$. Finally, we discuss the complexity of computing $\beta^{(t)}_k$ for $k < t$.  In order to compute these values efficiently throughout the algorithm's execution, we stored our normalized $\Q$-orthogonal partial basis consisting of the vectors $\dtbar[k]$, for all $k \in \{0, 1, \ldots, \T\}$. Since $\supp( \ut[t] ) = 1$, the cost of computing one $\beta^{(t)}_k$ is only $\bigo{\sparsity}$ and thus computing all them for $k < t$ and computing $\dt[t]$ takes $\bigo{\sparsity^2}$ operations.
    In summary, the time complexity of \cref{alg:sparse_conjugate_directions} is $\bigo{\sparsity^3 +\sparsity\vol(\suppast)}$. The space complexity of \cref{alg:sparse_conjugate_directions} is dominated by the cost of storing  $\dt[k]$ for $k \in \{0, 1, \ldots, t-1\}$, which is $\bigo{\sparsity^2}$.
\end{proof}

\input{outro_cd.tex}
} 

\begin{proof}\linkofproof{prop:apgd}
    The proof of the first part is derived from \citet[Theorem 4.10]{diakonikolas2019approximate}. The second part is a straightforward corollary. For any $\T \geq 1+\ceil{2\sqrt{\kappa}\log(\frac{(\L-\alpha)\norm{\x[0]-\xxast}^2}{2\epsilon})}$ we have
\begin{align*}
    f(\y[\T]) - f(\xast) & \leq \left(1 - \frac{1}{2\sqrt{\kappa}}\right)^{\T-1} \frac{(\L-\alpha)\norm{\xt[0]-\xxast}^2}{2} & \text{$\triangleright$ by the first part of \cref{prop:apgd}}\\
    & \leq \exp\left(-\frac{1}{2\sqrt{\kappa}}(\T-1)\right) \frac{(\L-\alpha)\norm{\xt[0]-\xxast}^2}{2} & \text{$\triangleright$ since $(1 + x) \leq e^x$ for all $x\in \R$}\\ 
    & \leq \epsilon
\end{align*}
    In particular, by $\alpha$-strong convexity of $f$, we have $\norm{\xt[0]-\xxast}^2 \leq \frac{\norm{\nabla f(\xt[0])}_2^2}{\alpha^2} $ so we obtain an $\epsilon$-minimizer after $1+\ceil{2\sqrt{\kappa}\log(\frac{(\L-\alpha)\norm{\nabla f(\xt[0])}_2^2}{2\epsilon\alpha^2})}$ iterations.
\end{proof}

\begin{proof}\linkofproof{lemma:A_positive}
    Let $i, j \in[\n]$ such that $i\neq j$. Geometrically, since the gradient of $\g$ is an affine function, the set of points $\yy$ for which $\nabla_j \g(\yy) \geq c$ for some value $c$, forms a halfspace. Fixing $x_i$, any point otherwise coordinatewise smaller than $\xx$ does not increase in gradient, since the off-diagonal entries of $\Q$ are non-positive. That is, the corresponding $(\n-1)$-dimensional halfspace is defined by a packing constraint \citep{allenzhu2019nearly, criado2021fast}. Formally, we have 
\begin{align*}
    \nabla_jg(\yy)-\nabla_jg(\xx) & = (\Q\yy)_j - (\Q\xx)_j = -\epsilon(\Q\canonical[i])_j = - \epsilon \Q_{j,i} \geq 0,
\end{align*}
    where the last inequality uses $\Q_{i,j} = -\frac{(1-\alpha)\A_{i,j}}{2d_i d_j} \leq 0$. The second statement is analogous.
\end{proof}

{ 
\input{definitions_aspr.tex}
\begin{proof}\linkofproof{prop:apgd_approach}
    Because we have $\Sinit=\emptyset$, $\xtast[-1] = \xtast[0] = \xt[0] = \zeroterm$, and by definition it is $\xast \in \Rp^{\n}$, we have that the first two properties hold trivially for $t=0$. \cref{property:apgd_approach_S_monotone} also holds for $t=0$. Indeed, if the set of known good indices does not expand, we have $\nabla \g(\xt[0]) \geq 0$ and so we have $\xt[0] = \zeroterm = \xast$, and thus $\xt[0]$ is an $\epsilon$-minimizer of $\g$ for any $\epsilon > 0$.

    We now prove the three properties inductively. Fix $t \in \{0, 1, \dots, \T-1\}$ and assume \cref{property:apgd_approach_positivity_grad,property:apgd_approach_x_monotone,property:apgd_approach_S_monotone} hold for this choice of $k \in \{0, \dots, t\}$. We will prove they hold for $t+1$.

    The value of the accuracy $\hatepsilon_t$ in \cref{alg:sparse_acceleration} was chosen to compute $\xtbar[t+1]$ close enough to $\xtast[t+1]$. In particular, we have
\begin{align}\label{eq:dist_to_opt_less_than_delta}
 \begin{aligned}
     \norm{\xtbar[t+1]-\xtast[t+1]}^2 \circled{1}[\leq] \frac{2}{\alpha} (\g(\xt[t+1]) - \g(\xtast[t+1])) \circled{2}[\leq] \frac{2\hatepsilon_t}{\alpha} \circled{3}[=] \delta_t^2,
\end{aligned}
\end{align}
    where we used $\alpha$-strong convexity of $\g$ for $\circled{1}$, the convergence guarantee of \apgd{} on $\xtbar[t+1]$ for $\circled{2}$, and for $\circled{3}$ we used the definition of $\hatepsilon_t$. The above allows to show that $\xt[t+1] \defi \max\{\zeroterm, \xtbar[t+1]-\delta_t \oneterm\}$, where the $\max$ is taken coordinatewise, satisfies 
\begin{equation}\label{eq:xt_leq_xtast}
    \xt[t+1] \leq \xtast[t+1].
\end{equation}
Suppose this property does not hold and that for some $i$ we have $\xt[t+1][i] > \xtast[t+1][i] \geq 0$. Then, we would have that $\xt[t+1][i] =\xtbar[t+1][i] - \delta_t$ and
    \[
        \norm{\xtbar[t+1]-\xtast[t+1]} \geq \abs{\xtbar[t+1][i]-\xtast[t+1][i]} \geq \xtbar[t+1][i]-\xtast[t+1][i] = \xt[t+1][i] + \delta_t -\xtast[t+1][i] > \delta_t,
    \] 
    which is a contradiction. Note that we have
    \begin{equation}\label{eq:neg_grad_at_xtast}
        \nabla_j \g(\xtast[t+1]) \circled{1}[\leq] \nabla_j \g(\xt[t+1]) \circled{2}[<] 0 \qquad \text{for all} \ j \in \St[t+1]\setminus \St[t],
    \end{equation}
    since $\circled{2}$ holds by definition of $\St[t+1]$ and $\circled{1}$ is due to \cref{lemma:A_positive} and the fact that we can write $\xt[t+1] = \xtast[t+1] - \sum_{i\in \St[t]}\omega_i \canonical[i]$ for some $\omega_i \in\Rp$, since we just proved $\xt[t+1] \leq \xtast[t+1]$ in \eqref{eq:xt_leq_xtast}, and by construction $\supp(\xt[t+1])\subseteq \St[t]$ and $\supp(\xtast[t+1]) \subseteq \St[t]$.

We now show that
    \begin{equation}\label{eq:xtast_leq_xtPlus1ast} 
        \xtast[t] \leq \xtast[t+1].
    \end{equation}
    This fact holds by Item \ref{property:pgd_helper_monotone} of \cref{proposition:pgd_helper} with starting point $\xtast[t]$ and $S \gets \St[t]$, which makes it $\xx^{(\ast, C)} \gets \xtast[t+1]$. The assumptions of \cref{proposition:pgd_helper} hold  since $\xtast[t] = 0$ if $i\in[\n]\setminus\St[t] \subseteq [\n]\setminus\St[t-1]$ by construction and we have $\nabla_i \g(\xtast[t]) = 0$ for $i\in\St[t-1]$ by induction hypothesis of \cref{property:apgd_approach_positivity_grad} and  $\nabla_i \g(\xtast[t]) < 0$ for $i\in\St[t] \setminus \St[t-1]$ by the same argument we provided to show \eqref{eq:neg_grad_at_xtast}. In the same context, we also use Item \ref{property:pgd_helper_positivity} of \cref{proposition:pgd_helper}, using the fact that $\nabla_i \g(\xtast[t]) < 0$ for $i\in\St[t] \setminus \St[t-1]$ and that by \cref{property:apgd_approach_positivity_grad} for $t$, we have $\xtast[t][i] > 0$ for all $i \in \St[t-1]$. Therefore, we conclude $\xtast[t+1][i] > 0$ for all $i\in \St[t]$, which means we proved \cref{property:apgd_approach_positivity_grad} for $t+1$.

    Moreover, now using Item \ref{property:pgd_helper_subset} of \cref{proposition:pgd_helper} in this context, we conclude
    \begin{equation}\label{eq:xtast_leq_xast}
      \xtast[t+1] \leq \xast.
    \end{equation}
    Thus, \cref{property:apgd_approach_x_monotone} holds for $t+1$ since we proved \eqref{eq:xt_leq_xtast}, \eqref{eq:xtast_leq_xtPlus1ast} and \eqref{eq:xtast_leq_xast}.

    Now we prove \cref{property:apgd_approach_S_monotone} for $t+1$. We note that the value of $\delta_t$ was chosen so that the retracted point $\xt[t+1]$ still enjoys a small enough gap:
    \begin{align} \label{eq:retracted_point_still_has_small_gap}
 \begin{aligned}
     \g(\xt[t+1] ) - \g(\xtast[t+1])  &\circled{1}[\leq] \frac{\L}{2} \norm{\xt[t+1] - \xtast[t+1]}_2^2   \\
        &\circled{2}[\leq] \L (\norm{\xtbar[t+1] - \xtast[t+1]}_2^2 + \card{\St[t]} \delta_t^2)  \\
        & \circled{3}[\leq] \L (1 + \card{\St[t]}) \delta_t^2  \\
        &\circled{4}[\leq] \frac{\epsilon\alpha}{\L}. 
   \end{aligned}
    \end{align}
    Above, $\circled{1}$ uses the optimality of $\xtast[t+1]$ and $\L$-smoothness, while $\circled{2}$ holds because by construction of $\xt[t+1]$, we have $\sum_{i\in \St[t]} \abs{\xt[t+1][i]-\xast[i]}^2 \leq \sum_{i\in \St[t]} (\abs{\xtbar[t+1][i]-\xast[i]} +\delta_t)^2 \leq 2 \sum_{i\in \St[t]} (\xtbar[t+1][i]-\xast[i])^2+2\card{ \St[t]} \delta_t^2$. We have $\circled{3}$ by \eqref{eq:dist_to_opt_less_than_delta} and $\circled{4}$ holds by the definition of $\delta_t$, which was made to satisfy this inequality.

    We now show that if $\xt[t+1]$ is not an $\epsilon$-minimizer of $\g$ in $\Rp^{\n}$, then one step of \pgd{} makes more progress than what can be made in $\Ct[t+1]$, by \eqref{eq:retracted_point_still_has_small_gap}, and so \pgd{} explores a new coordinate, that is, we have a coordinate $i$ with $\nabla_i \g(\xt[t+1]) < 0$ and we can extend the set of good coordinates $\St[t]$. Indeed, suppose that $\g(\xt[t+1]) - \g(\xast) > \epsilon$.  
    Let $\yy^{(t+1)}= \proj{\R^n_{\geq 0}}\left(\xt[t+1] - \nabla \g(\xt[t+1])\right)$.
    We use the following property in \citep[Equation below (23)]{fountoulakis2019variational} from the guarantees of \ista{} on the problem, or equivalently on $\pgd{}\left(\Ct[t+1], \xt[t], \g, 1\right)$, see \cref{sec:pgd}. We have
    \begin{align}\label{eq:projected_gd_guarantee_in_lemma}
        \g(\yy^{(t+1)}) - \g(\xast) \leq \left(1 - \frac{\alpha}{\L}\right) (\g(\xt[t+1]) - \g(\xast)).
    \end{align}
    Consequently, we obtain
    \begin{align*}
        \g(\xt[t+1]) - \g(\xtast[t+1]) &\circled{1}[\leq] \frac{\epsilon\alpha}{\L} \circled{2}[<] \frac{\alpha}{\L} (\g(\xt[t+1]) - \g(\xast)) \circled{3}[\leq] \g(\xt[t+1]) - \g(\yy^{(t+1)}),
    \end{align*}
    where $\circled{1}$ holds by \eqref{eq:retracted_point_still_has_small_gap}, $\circled{2}$ holds by our earlier assumption $\g(\xt[t+1])-\g(\xast) > \epsilon$, and $\circled{3}$ is obtained by \eqref{eq:projected_gd_guarantee_in_lemma} after adding $\g(\xt[t+1]) - \g(\yy^{(t+1)})$ to both sides, and reorganizing. Hence,
        $
        \g(\xtast[t+1]) > \g(\yy^{(t+1)}).
    $
    Since $\xtast[t+1]$ is the minimizer of $\g$ in $\Ct[t]$, it holds that $\yy^{(t+1)} \not \in \Ct[t]$ and so $\nabla_ig(\xt[t+1]) < 0 $ for at least one $i \not\in \St[t]$, and $\St[t] \subsetneq \St[t+1]$. 

    It remains to prove that $\St[t+1] \subseteq \suppast$. For $t+1 = \T$ it is $\St[\T-1] = \St[\T]$ and the property holds by induction hypothesis. For the case $t+1 \neq \T$, suppose the property does not hold and so there exists $j \not\in\St[t]$ such that $j\not\in\suppast$ and $\nabla_j \g(\xt[t+1]) < 0$. In that case, we have by \eqref{eq:neg_grad_at_xtast} that $\nabla_j \g(\xtast[t+1]) < 0$. On the other hand, it is $\nabla_i \g(\xtast[t+1]) = 0$ and $\xtast[t+1][i] > 0$ for $i \in \St[t]$ by \cref{property:apgd_approach_positivity_grad} and so we can apply \cref{proposition:pgd_helper} with $S\gets \St[t] \cup\{j\}$ and initial point $\xtast[t+1]$ to conclude a contradiction, since by Item \ref{property:pgd_helper_positivity} it is $x^{(\ast, C)}_j > 0$ but by Item \ref{property:pgd_helper_subset} we have  $x^{(\ast, C)}_j \leq \xast[j] =0$.

    Finally, by \cref{property:apgd_approach_S_monotone} for $t=\T$, since $\St[\T-1]$ does not expand, that is, $\St[\T-1] = \St[\T]$, it must be $\g(\xt[\T]) - \g(\xast) \leq \epsilon$.

\end{proof}

\begin{proof}\linkofproof{thm:apgd_approach_complexity}
    For each iteration, the time complexity of \cref{alg:sparse_acceleration} is the cost of the \apgd{} subroutine plus the full gradient computation in Line \ref{line:expanding_S_t}. By \cref{prop:apgd_approach}, \apgd{} is called at most $\T=\sparsity$ times and it runs for $\cO\left(\sqrt{\frac{\L}{\alpha}} \log\left(\frac{(\L-\alpha)\|\nabla_{\St[t]}\g(\xt[t])\|_2^2}{\hatepsilon_t\alpha^2}\right) \right)$ iterations at each stage $t$. One iteration of \apgd{} involves the computation of the gradient restricted to the current subspace of good coordinates, and involves the update of the iterates, costing $\bigo{\intvol(\suppast)}$. The computation of the full gradient takes $\bigo{\vol(\suppast)}$ operations.
    Thus, the total running time of \cref{alg:sparse_acceleration} is
    \begin{align*}
        &\bigol{\sparsity  \intvol(\suppast) \sqrt{\frac{\L}{\alpha}}\log\left(\frac{(\L-\alpha)\max_{t\in \{0, 1, \ldots, \T-1\}}\|\nabla_{\St[t]}\g(\xt[t])\|_2^2}{\hatepsilon_t\alpha^2} \right) + \sparsity \vol(\suppast)} \\
        & \circled{1}[=] \bigol{\sparsity  \intvol(\suppast) \sqrt{\frac{\L}{\alpha}}\log\left(\frac{\L^2(\L-\alpha)\norm{\zeroterm-\xast}^2}{\hatepsilon_t\alpha^2}\right) + \sparsity \vol(\suppast)}\\
        & = \bigotildel{\sparsity  \intvol(\suppast) \sqrt{\frac{\L}{\alpha}} + \sparsity \vol(\suppast)},
\end{align*}
    where $\circled{1}$ holds since by $\L$-smoothness of $\g$ restricted to $\spann{\{\canonical[i]\mid i\in\St[t]\}}$ and by $\zeroterm \leq \xt[t]\leq \xtast[t]\leq \xast$ for all $t\in [\T]$, we have 
$\norm{\nabla_{\St[t]}\g(\xt[t])}_2^2 \leq \L\norm{\xt[t] -\xtast[t]}_2^2 \leq \L\norm{\zeroterm -\xast}_2^2$.
    To further interpret the bound in the $\ell_1$-regularized PageRank problem, we can further bound 
\begin{align*}
     \norm{\zeroterm -\xast}_2^2 & \leq \frac{1}{\alpha^2} \norm{\nabla_{\suppast} \g (\zeroterm ) - \nabla_{\suppast} \g (\xast )}_2^2& \text{$\triangleright$ by $\alpha$-str. convexity of $\g$ in $\spann{\{\canonical[i]\mid i\in\suppast\}}$}\\
    & \leq \frac{1}{\alpha^2} \norm{\nabla_{\suppast} \g (\zeroterm )}_2^2 & \text{$\triangleright$ by optimality of $\xast$}\\
    & \leq \frac{1}{\alpha^2}\norm{(-\alpha \D^{-1/2}\ss + \alpha\rho \D^{1/2}\oneterm)_{\suppast}}_2^2 & \text{$\triangleright$ by the gradient definition}\\
    &\leq \frac{1}{\alpha^2}\norm{(-\D^{-1/2}\oneterm+\D^{1/2}\oneterm)_{\suppast}}_2^2 & \text{$\triangleright$ $\alpha, \ss[i], \rho \leq 1$}\\
    &\leq \frac{1}{\alpha^2}(1 + \sqrt{\vol(\suppast)})^2 \sparsity& \text{$\triangleright$ maximum $d_i$ for $i\in\suppast$ is $\leq \card{\vol(\suppast)}$}\\
    &= \bigol{\frac{1}{\alpha^2}\abs{\suppast}\vol(\suppast)}.
\end{align*}
    Then, by the definition of $\hatepsilon_t$, the time complexity of \cref{alg:sparse_acceleration} of the $\ell_1$-regularized PageRank problem is
    \begin{align*}
        &\bigol{| \suppast |\intvol(\suppast)\sqrt{\frac{\L}{\alpha}}\log\left(\frac{2L^4 (1+\card{\St[t]})(\L-\alpha) | \suppast |  \vol(\suppast)}{\alpha^6\epsilon}\right) + \sparsity\vol(\suppast)}.
\end{align*}

    The space complexity of \cref{alg:sparse_acceleration} is dominated by the cost of storing  the gradient $\nabla_{\St[t]} \g(\xt[t])$, which is $\bigo{\sparsity}$, since $\St[t] \subseteq \suppast$. Note that we require to compute the full gradient when updating $\St[t+1]$, but we only store the new indices.
\end{proof}

\input{outro_aspr.tex}
}

\begin{proof}\linkofproof{lemma:modification}
    Fix $j \not\in S$ such that $\nabla_j \g(\xx) < 0$. By the assumption $x_i = 0$ if $i\not\in S$ and $\nabla_i \g(\xx) \leq 0$ if $i\in S$, and therefore we can use \cref{proposition:pgd_helper} with $S$ and $\xx$ to conclude $\zeroterm \leq \xx \leq \xtast[C]$ and $\nabla_j \g(\xtast[C]) = 0$. We can thus write $\xx = \xtast[C] - \sum_{i\in S}\omega_i \canonical[i]$, for $\omega_i \in\Rp$ for all $i\in S$. By \cref{lemma:A_positive}, it holds that $\nabla_j \g(\xtast[C]) \leq \nabla_j \g(\xx) < 0$. We now use \cref{property:pgd_helper_positivity,property:pgd_helper_subset} in \cref{proposition:pgd_helper} with set of indices $\bar{S}\defi S \cup \{j\}$ and starting point $\xtast[C]$. By \cref{property:pgd_helper_positivity} we have for $\xtast[\bar{C}]=\argmin_{\xx\in\bar{C}}\g(\xx)$ that $\xtast[\bar{C}][j] > 0$, where $\bar{C} \defi \spann{\{\canonical[i] \mid i \in \bar{S}\}}$. By \cref{property:pgd_helper_monotone}, we have $\xtast[\bar{C}][i] \geq \xtast[C][i] > 0$ for $i \in S$. Thus, $\xtast[\bar{C}][i]> 0$ for all $i\in\bar{S}$ and by \cref{property:pgd_helper_subset}, it holds $\bar{S} \subseteq \suppast$ and in particular $j\in\suppast$.
\end{proof}

\section{Algorithmic Comparisons}\label{sec:algorithmic_comparisons}
\cdappr{} has worse space complexity, $\cO(\sparsity^2)$, than \aspr{} and \ista{}, both $\cO(\sparsity)$. However, since \cdappr{} finds the exact solution, \cdappr{} outperforms the other methods in running time for small enough $\epsilon$. Note that the time complexities of \ista{} and \aspr{} depend on $\frac{1}{\epsilon}$ only logarithmicly. We perform the remaining comparison for $\log(1/\epsilon)$ treated as a constant, since it is so in practice.
If 
\begin{align*}
     \frac{\L}{\alpha} > \max\left\{\frac{\sparsity^3}{\volast}, \sparsity \right\},
\end{align*}
then \cdappr{} performs better than \ista{}, up to constants. Since $\volast\geq \sparsity$, this is, for example, satisfied when $\frac{\L}{\alpha} > \sparsity^2$. 
If
\begin{align*}
    \frac{\L}{\alpha} > \max\left\{\left(\frac{\sparsity\intvol(\suppast)}{\volast}\right)^2, \sparsity \right\},
\end{align*}
then \aspr{} performs better than \ista{}, up to constants and log factors. This is, for example, satisfied when $\frac{\L}{\alpha} > \sparsity^2$ or when $\frac{\L}{\alpha} > \sparsity$ and $\volast > \sparsity^{5/2}$ since $\intvol(\suppast) \leq \sparsity^2$.
If the convergence rates of \cdappr{} and \aspr{} are dominated by $\cO(\sparsity \volast)$, then the algorithms perform similarly. However, if the time complexities of \cdappr{} and \aspr{} are of orders $\cO(\sparsity^3)$ and $\cO(\sparsity \volast \sqrt{\frac{\L}{\alpha}})$, respectively, then \cdappr{} performs better than \aspr{} for
\begin{align*}
    \frac{\L}{\alpha} > \left(\frac{\sparsity^2}{\intvol(\suppast)}\right)^2,
\end{align*}
up to constants and log factors. We note that although \citet{fountoulakis2019variational} describe their method as using $\bigo{\vol(\suppast)}$ memory, their \ista{} solver actually only requires $\bigo{\sparsity}$ space, as it is enough to store the entries of the iterates and gradients corresponding to the good coordinates, whereas the gradient entries for bad coordinates can be discarded immediately after computation.

\end{document}

%% file: prologueCOLT.tex
\usepackage[utf8]{inputenc} 
\usepackage[T1]{fontenc}    

\usepackage{url}            
\usepackage{booktabs}       
\usepackage{amsfonts}       
\usepackage{microtype}      

\usepackage{nicefrac}

\usepackage{amsmath}

\usepackage{amssymb}
\usepackage{thmtools,thm-restate}
\usepackage{thm-restate}
\usepackage{paralist}
\usepackage{comment}
\usepackage{dsfont}

\usepackage{bbm}
\usepackage{float}
\usepackage{balance}
\usepackage{stmaryrd}

\usepackage{multicol}
\usepackage[export]{adjustbox}
\usepackage{multirow}

\usepackage{mathtools}
\usepackage{mdframed}

\usepackage{enumitem}
\setitemize{noitemsep,topsep=0pt,parsep=0pt,partopsep=0pt}

\usepackage{xcolor} 
\usepackage{tikz}
\usetikzlibrary{fit,calc}

\colorlet{pink}{red!40}
\colorlet{lightblue}{blue!30}
\colorlet{lightgreen}{green!30}

\DontPrintSemicolon

\renewcommand{\cite}[1]{\citep{#1}}

\usepackage[vvarbb]{newtxmath}

\def\appr{{\textnormal{\texttt{APPR}}}\xspace}

\def\cdappr{{\textnormal{\texttt{CDPR}}}\xspace}
\def\ista{{\textnormal{\texttt{ISTA}}}\xspace}
\def\fista{{\textnormal{\texttt{FISTA}}}\xspace}

\def\bb{{\mathbf{b}}}

\def\dd{{\mathbf{d}}}
\def\ee{{\mathbf{e}}}

\def\ss{{\mathbf{s}}}

\def\uu{{\mathbf{u}}}

\def\xx{{\mathbf{x}}}
\def\yy{{\mathbf{y}}}
\def\zz{{\mathbf{z}}}

\newcommand{\D}{\mathbb{D}}

\newcommand{\N}{\mathbb{N}}

\newcommand{\Q}{\mathbb{Q}}
\newcommand{\R}{\mathbb{R}}
\newcommand{\T}{\mathbb{T}}

\newcommand{\V}{\mathbb{V}}

\newcommand\cL{{\ensuremath{\mathcal{L}}}\xspace}

\newcommand\cO{{\ensuremath{\mathcal{O}}}\xspace}

\newcommand\cS{{\ensuremath{\mathcal{S}}}\xspace}

\DeclareMathOperator{\supp}{supp}
\DeclareMathOperator{\mathspan}{span}

\DeclareMathOperator{\vol}{vol}

\newcommand{\ones}{\mathbf{1}}
\newcommand{\zeroterm}{\ensuremath{\mathbb{0}}}
\newcommand{\oneterm}{\ensuremath{\mathbb{1}}}

\newcommand*{\vsepfbox}[1]{%
  \begingroup
    \sbox0{\fbox{#1}}%
    \setlength{\fboxrule}{0pt}%
    \mbox{\kern-\fboxsep\fbox{\unhbox0}\kern-\fboxsep}%
  \endgroup
}

\graphicspath{{imgs/}}
\makeatletter
\def\mathcolor#1#{\@mathcolor{#1}}
\def\@mathcolor#1#2#3{%
  \protect\leavevmode
  \begingroup
    \color#1{#2}#3%
  \endgroup
}
\makeatother

%% file: make_cleveref_work.tex
\usepackage[capitalise,nameinlink,noabbrev]{cleveref} % Cite with \cref of the object (Theorem, Proposition, etc.) is written automatically. Place before custom theorem declarations so it works with them

\makeatletter
  \newcommand{\jmlrBlackBox}{\rule{1.5ex}{1.5ex}}
  
  \newcommand{\jmlrQED}{\hfill\jmlrBlackBox\par\bigskip}
\providecommand{\proofname}{Proof}
  \newenvironment{proof}%
  {%
   \par\noindent{\bfseries\upshape \proofname\ }%
  }%
  {\jmlrQED}
  \newcommand*{\theorembodyfont}[1]{%
    \renewcommand*{\@theorembodyfont}{#1}%
  }
  \newcommand*{\@theorembodyfont}{\normalfont\itshape}%
  \newcommand*{\theoremheaderfont}[1]{%
    \renewcommand*{\@theoremheaderfont}{#1}%
  }
  \newcommand*{\@theoremheaderfont}{\normalfont\bfseries }%
  \newcommand*{\theoremsep}[1]{%
    \renewcommand*{\@theoremsep}{#1}%
  }
  \newcommand*{\@theoremsep}{}%
  \newcommand*{\theorempostheader}[1]{%
    \renewcommand*{\@theorempostheader}{#1}%
  }
  \newcommand*{\@theorempostheader}{}%
  \let\jmlr@org@newtheorem\newtheorem
  \renewcommand*{\newtheorem}{\@ifstar\jmlr@snewtheorem\jmlr@newtheorem}
  \newcommand*{\jmlr@snewtheorem}[2]{%
    \cslet{jmlr@thm@#1@body@font}{\@theorembodyfont}%
    \cslet{jmlr@thm@#1@header@font}{\@theoremheaderfont}%
    \cslet{jmlr@thm@#1@sep}{\@theoremsep}%
    \cslet{jmlr@thm@#1@postheader}{\@theorempostheader}%
    \newenvironment{#1}%
    {%
      \trivlist
        \item
        [%
          \hskip\labelsep{\csuse{jmlr@thm@#1@header@font}#2%
            \csuse{jmlr@thm@#1@postheader}%
          }%
        ]%
        \mbox{}\csuse{jmlr@thm@#1@sep}%
        \csuse{jmlr@thm@#1@body@font}%
    }%
    {%
      \endtrivlist
    }%
  }
  \newcommand{\jmlr@newtheorem}[1]{%
    \cslet{jmlr@thm@#1@body@font}{\@theorembodyfont}%
    \cslet{jmlr@thm@#1@header@font}{\@theoremheaderfont}%
    \cslet{jmlr@thm@#1@sep}{\@theoremsep}%
    \cslet{jmlr@thm@#1@postheader}{\@theorempostheader}%
    \jmlr@org@newtheorem{#1}%
  }
  \renewcommand*{\@xthm}[2]{%
    \def\@jmlr@currentthm{#1}%
    \@begintheorem{#2}{\csname the#1\endcsname}%
    \ignorespaces
  }
  \def\@ythm#1#2[#3]{%
    \def\@jmlr@currentthm{#1}%
    \@opargbegintheorem{#2}{\csname the#1\endcsname}{#3}%
    \ignorespaces
  }
  \renewcommand*{\@begintheorem}[2]{%
    \ifdef{\@jmlr@currentthm}%
    {%
      \letcs{\jmlr@this@theoremheader}{jmlr@thm@\@jmlr@currentthm @header@font}%
      \letcs{\jmlr@this@theorembody}{jmlr@thm@\@jmlr@currentthm @body@font}%
      \letcs{\jmlr@this@theoremsep}{jmlr@thm@\@jmlr@currentthm @sep}%
      \letcs{\jmlr@this@theorempostheader}%
         {jmlr@thm@\@jmlr@currentthm @postheader}%
    }%
    {%
      \let\jmlr@this@theorembody\@theorembodyfont
      \let\jmlr@this@theoremheader\@theoremheaderfont
      \let\jmlr@this@theoremsep\@theoremsep
      \let\jmlr@this@theorempostheader\@theorempostheader
    }%
    \trivlist
      \item
       [%
        \hskip\labelsep{\jmlr@this@theoremheader #1\ #2%
           \jmlr@this@theorempostheader}%
       ]%
      \mbox{}\jmlr@this@theoremsep
      \jmlr@this@theorembody
  }
  \renewcommand*{\@opargbegintheorem}[3]{%
    \ifdef{\@jmlr@currentthm}%
    {%
      \letcs{\jmlr@this@theoremheader}{jmlr@thm@\@jmlr@currentthm @header@font}%
      \letcs{\jmlr@this@theorembody}{jmlr@thm@\@jmlr@currentthm @body@font}%
      \letcs{\jmlr@this@theoremsep}{jmlr@thm@\@jmlr@currentthm @sep}%
      \letcs{\jmlr@this@theorempostheader}%
         {jmlr@thm@\@jmlr@currentthm @postheader}%
    }%
    {%
      \let\jmlr@this@theorembody\@theorembodyfont
      \let\jmlr@this@theoremheader\@theoremheaderfont
      \let\jmlr@this@theoremsep\@theoremsep
      \let\jmlr@this@theorempostheader\@theorempostheader
    }%
    \trivlist
     \item[\hskip\labelsep{\jmlr@this@theoremheader #1\ #2\ (#3)%
       \jmlr@this@theorempostheader}]%
     \mbox{}\jmlr@this@theoremsep
     \jmlr@this@theorembody
  }
\makeatother

%% file: algorithm_config.tex
\usepackage{algorithm}

\usepackage{algcompatible}
\algnewcommand{\lst}{\texttt{lst}}
\algnewcommand{\slst}{\texttt{slst}}
\algnewcommand{\SEND}{\textbf{send}}

\newsavebox{\algleft}
\newsavebox{\algright}

\makeatletter
\newcounter{algorithmicH}
\let\oldalgorithmic\algorithmic
\renewcommand{\algorithmic}{%
  \stepcounter{algorithmicH}
  \oldalgorithmic}
\renewcommand{\theHALG@line}{ALG@line.\thealgorithmicH.\arabic{ALG@line}}
\makeatother

%% file: bibliography_config.tex
\usepackage[natbib, backend=biber, maxcitenames=3, minalphanames=3, maxbibnames=99, style=alphabetic, hyperref, backref, useprefix=true, uniquename=false, doi=false,url=false,eprint=false]{biblatex}

\usepackage{csquotes}               

\bibliography{bibliography}

\newbibmacro{string+doiurlisbn}[1]{%
  \iffieldundef{doi}{%
    \iffieldundef{url}{%
      \iffieldundef{isbn}{%
        \iffieldundef{issn}{%
          #1%
        }{%
          \href{http://books.google.com/books?vid=ISSN\thefield{issn}}{#1}%
        }%
      }{%
        \href{http://books.google.com/books?vid=ISBN\thefield{isbn}}{#1}%
      }%
    }{%
      \href{\thefield{url}}{#1}%
    }%
  }{%
    \href{https://doi.org/\thefield{doi}}{#1}%
  }%
}

\DeclareFieldFormat{title}{\usebibmacro{string+doiurlisbn}{\mkbibemph{#1}}}
\DeclareFieldFormat[article,incollection,inproceedings]{title}%
    {\usebibmacro{string+doiurlisbn}{\mkbibquote{#1}}}

%% file: definitions_cd.tex
\renewcommand\T{\newlink{def:final_iteration_T_of_conjugate_directions}{T}}

\let\oldxtast\xtast
\let\oldxt\xt
\let\oldSt\St
\let\oldCt\Ct
\let\oldSinit\Sinit

\let\xtast\undefined
\let\xt\undefined
\let\St\undefined
\let\Ct\undefined
\let\Sinit\undefined

\NewDocumentCommand{\xtast}{oo}{%  % Different behavior depending on whether the last arguments are provided or not
    % Usage as \xt
    \newlink{def:optimizer_in_subspace_for_CDPR}{
        \IfNoValueTF{#1}
            {\xx^{(\ast, t)}}
        {
        \IfNoValueTF{#2}
            {\xx^{(\ast, #1)}}
            {x^{(\ast, #1)}_{#2}}
        }
    }
}

\NewDocumentCommand{\xt}{oo}{%  % Different behavior depending on whether the last arguments are provided or not
    % Usage 
    %   1. \xt 
    %   2. \xt[t+1]
    %   3. \xt[t][i] but not \xt[][i]
    \newlink{def:iterate_xt_of_CDPR}{
        \IfNoValueTF{#1}{
            {\xx^{(t)}}
        }
        {
        \IfNoValueTF{#2}
            {\xx^{(#1)}}
            {x^{(#1)}_{#2}}
        }
    }
}

\newcommand\Sinit{\newlink{def:initial_set_of_known_good_coordinates_CDPR}{S^{( -1 )}}}
\newcommandx*\St[1][1=t, usedefault]{\newlink{def:set_of_known_good_coordinates_CDPR}{S^{( #1 )}}}
\newcommandx*\Ct[1][1=t, usedefault]{\newlink{def:span_of_known_good_coordinates_in_Rp_CDPR}{C^{( #1 )}}}

%% file: outro_cd.tex
\let\xtast\oldxtast
\let\xt\oldxt
\let\St\oldSt
\let\Ct\oldCt
\let\Sinit\oldSinit

%% file: definitions_aspr.tex
\renewcommand\T{\newlink{def:final_iteration_T_of_ASPR}{T}}

\let\oldxtast\xtast
\let\oldxt\xt
\let\oldSt\St
\let\oldCt\Ct
\let\oldSinit\Sinit

\let\xtast\undefined
\let\xt\undefined
\let\St\undefined
\let\Ct\undefined
\let\Sinit\undefined

\NewDocumentCommand{\xtast}{oo}{%  % Different behavior depending on whether the last arguments are provided or not
    % Usage as \xt
    \newlink{def:optimizer_in_subspace_for_ASPR}{
        \IfNoValueTF{#1}
            {\xx^{(\ast, t)}}
        {
        \IfNoValueTF{#2}
            {\xx^{(\ast, #1)}}
            {x^{(\ast, #1)}_{#2}}
        }
    }
}

\NewDocumentCommand{\xt}{oo}{%  % Different behavior depending on whether the last arguments are provided or not
    % Usage 
    %   1. \xt 
    %   2. \xt[t+1]
    %   3. \xt[t][i] but not \xt[][i]
    \newlink{def:iterate_xt_of_ASPR}{
        \IfNoValueTF{#1}{
            {\xx^{(t)}}
        }
        {
        \IfNoValueTF{#2}
            {\xx^{(#1)}}
            {x^{(#1)}_{#2}}
        }
    }
}

\newcommand\Sinit{\newlink{def:initial_set_of_known_good_coordinates_ASPR}{S^{( -1 )}}}
\newcommandx*\St[1][1=t, usedefault]{\newlink{def:set_of_known_good_coordinates_ASPR}{S^{( #1 )}}}
\newcommandx*\Ct[1][1=t, usedefault]{\newlink{def:span_of_known_good_coordinates_in_Rp_ASPR}{C^{( #1 )}}}

%% file: outro_aspr.tex
\let\xtast\oldxtast
\let\xt\oldxt
\let\St\oldSt
\let\Ct\oldCt
\let\Sinit\oldSinit